\documentclass[11pt]{amsart}
\usepackage{amsthm,amsmath,amstext,amssymb,amscd,euscript, mathrsfs, dsfont,multicol,,times,enumerate,subfig,sidecap}
\usepackage{amsmath}
\usepackage{amssymb}
\usepackage{bbm}

\usepackage[colorlinks=true, pdfstartview=FitV, linkcolor=blue, citecolor=red, urlcolor=blue]{hyperref} 

\newcommand{\R}{{\mathbb R}}       
\newcommand{\Z}{{\mathbb Z}}       

\newcommand{\cM}{{\mathcal M}}

\newcommand{\cA}{{\mathcal A}}

\newcommand{\cS}{{\mathcal S}}

\newcommand{\diam}{\mathop{\rm diam}}
\newcommand{\dist}{{\rm dist}}

\newcommand{\supp}{\operatorname{supp}}

\newcommand{\wh}{\widehat}
\newcommand{\meas}{\mathrm{meas}}
\newcommand{\rF}{\mathrm{F}}

\def\XXint#1#2#3{{\setbox0=\hbox{$#1{#2#3}{\int}$ }
\vcenter{\hbox{$#2#3$ }}\kern-.58\wd0}}

\newtheorem{theorem}{Theorem}[section]
\newtheorem{lemma}[theorem]{Lemma}
\newtheorem{remark}[theorem]{Remark}
\newtheorem{corollary}[theorem]{Corollary}

\newtheorem{proposition}[theorem]{Proposition}

\newtheorem*{lemma*}{Lemma}
\newtheorem*{theorem*}{Theorem}

\theoremstyle{definition}

\theoremstyle{remark}
\newtheorem{rem}[theorem]{\bf Remark}

\numberwithin{equation}{section}

\newcommand{\brem}{\begin{rem}}
\newcommand{\erem}{\end{rem}}

\makeatletter
\def\@tocline#1#2#3#4#5#6#7{\relax
  \ifnum #1>\c@tocdepth 
  \else
    \par \addpenalty\@secpenalty\addvspace{#2}%
    \begingroup \hyphenpenalty\@M
    \@ifempty{#4}{%
      \@tempdima\csname r@tocindent\number#1\endcsname\relax
    }{%
      \@tempdima#4\relax
    }%
    \parindent\z@ \leftskip#3\relax \advance\leftskip\@tempdima\relax
    \rightskip\@pnumwidth plus4em \parfillskip-\@pnumwidth
    #5\leavevmode\hskip-\@tempdima
      \ifcase #1
       \or\or \hskip 1em \or \hskip 2em \else \hskip 3em \fi%
      #6\nobreak\relax
    \dotfill\hbox to\@pnumwidth{\@tocpagenum{#7}}\par
    \nobreak
    \endgroup
  \fi}
\makeatother

\begin{document}

\title[$L^p$ improving property and maximal multilinear averages]{$L^p$ improving properties and maximal estimates for certain multilinear averaging operators}

\author[C. Cho]{Chu-hee Cho}

\address{Chuhee Cho
\\
Research Institute of Mathematics
\\
Seoul National University
\\
08826 Gwanak-ro 1, Seoul, Republic of Korea} \email{akilus@snu.ac.kr}

\author[J. B. Lee]{Jin Bong Lee}

\address{Jin Bong Lee
\\
Research Institute of Mathematics
\\
Seoul National University
\\
08826 Gwanak-ro 1, Seoul, Republic of Korea} \email{jinblee@snu.ac.kr}

\author[K. Shuin]{Kalachand Shuin}

\address{Kalachand Shuin
\\
Department of Mathematical Sciences
\\
Seoul National University
\\
08826 Gwanak-ro 1, Seoul, Republic of Korea} \email{kcshuin21@snu.ac.kr}

\begin{abstract}
In this article we focus on $L^{p}$ estimates for two types of multilinear lacunary maximal averages over hypersurfaces with curvature conditions. 
Moreover, we give a different proof for the bilinear lacunary spherical maximal functions.
To obtain our results, we make use of the $L^1$-improving estimates of multilinear averaging operators.
We also obtain $L^p$-improving estimates for certain multilinear averages by means of the nonlinear Brascamp-Lieb inequality.
\end{abstract}


\maketitle

\tableofcontents

\section{Introduction}

Let $\mathcal{S}$ be a compact and smooth hypersurface contained in a unit ball $\mathbb{B}^{d}(0,1)$ with $\kappa$ non-vanishing principal curvatures,
and $\Theta_j$ be rotation matrices in $\mathbf{M}_{d,d}(\R)$ for $j=1,2, \dots, m$.
We assume that $\{\Theta_j\}_{j=1}^m$ is mutually linearly independent. 
Then for  $f_{1},f_{2},\dots,f_{m}\in \mathscr{S}(\mathbb{R}^{d})$, we define  
\begin{align}\
  \mathcal{A}_{\mathcal{S}}^{\Theta}(\rF)(x) &:= \int_{\mathcal{S}} \prod_{j=1}^m f_j(x+\Theta_j y)~\mathrm{d}\sigma_{\mathcal{S}}(y)\label{defn-mlao},
  \end{align}
where $F=(f_1,f_2,\cdots,f_m)$ and $d\sigma_{\mathcal{S}}$ is the normalized surface measure on $\mathcal{S}$. 
We also consider another $m$-linear averaging operator defined by 
\begin{align}
  \textsl{A}_{\Sigma}(\rF)(x) &:= \int_{\Sigma} \prod_{j=1}^m f_j (x + y_j)~\mathrm{d}\sigma_{\Sigma}(\textsl{y}),\,\, (y_1, \dots, y_m) = \textsl{y} \in \R^{md},\label{defn-mlao-cM}
\end{align}
where $\Sigma$ is a compact $(md-1)$ dimensional smooth hypersurface contained in a unit ball $\mathbb{B}^{md}(0,1)$ with $\kappa$ non-vanishing principal curvatures. 
Note that $\kappa$ arising in \eqref{defn-mlao} satisfies $1\leq \kappa\leq d-1$, while $\kappa$ in \eqref{defn-mlao-cM} is $1\leq \kappa\leq md-1$.
Moreover, we are interested in the following lacunary maximal operators associated with \eqref{defn-mlao} and \eqref{defn-mlao-cM}: 
\begin{align}
	\mathcal{M}^{\Theta}_{\mathcal{S}}(\mathrm{F})(x) &= \sup_{\ell\in\Z} \Big|\int_{\mathcal{S}}\prod_{j=1}^m f_j(x - 2^\ell \Theta_{j}y)~~\mathrm{d}\sigma_{\mathcal{S}}(y) \Big|,\label{maximal avg 1}\\
	\mathfrak{M}_{\Sigma}(\mathrm{F})(x) &= \sup_{\ell\in\Z} \Big|\int_{\Sigma}\prod_{j=1}^m f_j(x - 2^\ell y_j)~~\mathrm{d}\sigma_\Sigma(\textsl{y}) \Big|.\label{maximal avg 2}
\end{align}
The purpose of this article is to prove $L^{p}$-improving estimates of multilinear averaging operators defined by \eqref{defn-mlao} and \eqref{defn-mlao-cM}.
Further, using this $L^{p}$-improving estimates we show $L^{p_{1}}\times L^{p_{2}}\times\cdots\times L^{p_{m}}\rightarrow L^{p}$ boundedness for $1/p=\sum^{m}_{j=1}1/p_j$ of the multi-(sub)linear lacunary maximal functions $\cM_\cS^\Theta$ and $\mathfrak{M}_\Sigma$.

Averaging operators given in \eqref{defn-mlao} and \eqref{defn-mlao-cM} and related maximal operators arise in many studies in multilinear harmonic analysis. 
Since Coifman and Meyer \cite{CoifmanMeyer} opened the path of multilinear harmonic analysis in 1975,
there have been significant developments in the area of multilinear harmonic analysis over the last few decades.
Among those achievements, we introduce works of Lacey and Thiele \cite{LaceyTheile1, LaceyTheile2} in which they proved $L^{p}$-boundedness of the bilinear Hilbert transform given as
$$
	BHT_\alpha(f,g)(x) := p.v. \int_{-\infty}^\infty f(x-t)g(x-\alpha t)\frac{\mathrm{d}t}{t},\quad \alpha\not=0,1.
$$
Their seminal work settled the long standing conjecture of Calder\'{o}n. 
Later, Lacey \cite{Lacey2} studied $L^{p}$- boundedness of bilinear maximal operator  
$$
	M_\alpha(f,g)(x) := \sup_{t>0} \frac1{2t}\int_{-t}^t |f(x-y) g(x-\alpha y)|~\mathrm{d}y,\quad \alpha\not=0,1,
$$
which is related to the bilinear Hilbert transform.
One may regard averaging operators $\cA_\cS^\Theta$ as a generalization of $M_\alpha$ without the supremum because the condition $\alpha\not=0,1$ corresponds to the linearly independent condition of $\{\Theta_j\}$.

On the other hand, $\textsl{A}_\Sigma$ given in \eqref{defn-mlao-cM} is a direct analogue of a spherical averages $A^{t}_{\mathbb{S}^{d-1}}f(x)$ for $t=1$, which is defined by
$$
	A^{t}_{\mathbb{S}^{d-1}}f(x) := \int_{\mathbb{S}^{d-1}}f(x-ty)~\mathrm{d}\sigma(y).
$$
Thus we write $\textsl{A}_{\mathbb{S}^{md-1}}(\rF)(x) =  A^{1}_{\mathbb{S}^{md-1}}(f_1 \otimes\cdots\otimes f_m) (x,\dots, x)$.
For studies on $\textsl{A}_{\mathbb{S}^{md-1}}(\rF)$, we recommend \cite{Oberlin, Bak, SaurabhShuin, Dosidis} and references therein.
In the literature, $A^{t}_{\mathbb{S}^{d-1}}$ have been extensively studied in terms of maximal operators.
For the (sub)linear spherical maximal operator $M^{*}_{\mathbb{S}^{d-1}}$ defined by 
$$
	M^{*}_{\mathbb{S}^{d-1}}f(x)= \sup_{t>0}|A^{t}_{\mathbb{S}^{d-1}}f(x)| := \sup_{t>0}\Big|\int_{\mathbb{S}^{d-1}}f(x-ty)~\mathrm{d}\sigma(y)\Big|
$$
with $d\sigma$ is the normalized surface measure on the sphere $\mathbb{S}^{d-1}$, Stein \cite{Stein1976} proved that for $d\geq3$, the spherical maximal operator $M^{*}_{\mathbb{S}^{d-1}}$ is bounded in $L^{p}$, if and only if $p>\frac{d}{d-1}$ in 1976. 
Later, 
Bourgain \cite{Bourgain1986} obtained $L^{p}$ boundedness of $M^{*}_{\mathbb{S}^{1}}$ for $p>2$. 
Those restricted boundedness of $M^{*}_{\mathbb{S}^{d-1}}$ can be improved if one considers the lacunary spherical maximal operator, which is given by 
$M_{\mathbb{S}^{d-1}}f(x):=\sup_{j\in\mathbb{Z}}|A_{\mathbb{S}^{d-1}}^{2^{j}}f(x)|$.
Calder\'{o}n  \cite{Calderon} proved $L^{p}$ estimates of the operator $M_{\mathbb{S}^{d-1}}$ for $1<p 
\leq\infty$ and $d\geq2$. 
After then, Seeger and Wright \cite{Seegerwright} showed  $L^{p}$ estimates of general lacunary maximal operators $M_{\mathcal{S}}$ for $1<p 
\leq\infty$, when  the Fourier transform of the surface measure $\sigma$ of $\mathcal{S}$ satisfies $|\hat{\sigma}(\xi)|\lesssim |\xi|^{-\epsilon}$, for any $\epsilon>0$.
There are also $L^{p}- L^{q}$ estimates for  $p\leq q$ (we call this $L^{p}$-improving estimates) of the spherical average $A^{1}_{\mathbb{S}^{d-1}}$ \cite{Littman,Stichartz}.  

Lacey \cite{Lacey} used the $L^{p}$-improving estimates of spherical averages to prove sparse domination of the corresponding lacunary and full spherical maximal functions. 
It is well known that sparse domination of an operator implies vector valued boundedness and weighted boundedness of that operator with respect to Muckenhoupt $A_{p}$ weights \cite{Nieraeth, Lerner1}. 
This idea has been extensively used to obtain sparse domination of several linear and sub-linear operators in the field of Harmonic analysis (See \cite{Beltran1}). 
The idea of Lacey \cite{Lacey}, together with $L^{p}$-improving estimates of certain bilinear averaging operators, can be used to study sparse domination of maximal operators associated with the bilinear operators.
We recommend  \cite{JillPipher, Palsson,Luzsaurabh} and references therein, which contains results of bilinear spherical maximal operator, bilinear maximal triangle averaging  operators and bilinear product-type spherical maximal operators, respectively.

Recently, Christ and Zhou \cite{Christ_Zhou2022} studied $L^{p_{1}}\times L^{p_{2}}\rightarrow L^{p}$ with $1/p_1+1/p_2=1/p$ boundedness of bi-(sub)linear lacunary maximal functions defined on a class of singular curves, which might be understood in the sense of both \eqref{maximal avg 1} and \eqref{maximal avg 2}. 
$$
	\mathcal{M}(f_1, f_2)(x) := \sup_{\ell\in\mathbb{Z}}|B_{2^\ell}(f_1, f_2)(x)| = \sup_{\ell\in\mathbb{Z}}\Big|\int_{\R^1} \prod_{j=1}^2 f_j(x- 2^\ell \gamma_j(t)) \eta(t)~\mathrm{d}t \Big|,
$$
where $\gamma = (\gamma_1, \gamma_2) : (-1, 1) \to \R^2$, and $\eta \in C_0^\infty((-1, 1))$.
In consequence, they have proved $L^{p_{1}}\times L^{p_{2}}\rightarrow L^{p}$ estimates for $1<p_1,p_2\leq\infty$, $1/p_1+1/p_2=1/p$  of the bi-(sub)linear  lacunary spherical maximal operator $\mathfrak{M}_{\mathbb{S}^{2d-1}}$, for dimension $d=1$ which is defined by 
$$
	\mathfrak{M}_{\mathbb{S}^{1}}(f_{1},f_{2})(x)
	:=\sup_{\ell\in\mathbb{Z}}\Big| \textsl{A}_{\mathbb{S}^{1}}^{2^{\ell}}(f_{1},f_{2})(x)\Big|
	=\sup_{\ell\in\mathbb{Z}}\Big|\int_{\mathbb{S}^{1}} \prod_{j=1,2} f_{j}(x-2^{\ell}y_j)~\mathrm{d}\sigma(\textsl{y})\Big|,
$$
where $d\sigma(\textsl{y})$ is the normalized surface measure on the circle $\mathbb{S}^{1}$. 
For $d\ge2$, the complete $(L^{p_{1}}\times L^{p_{2}}\rightarrow L^{p})$-estimate of the operator $\mathfrak{M}_{\mathbb{S}^{2d-1}}$ was not known. 
However, there are some partial results of the operator $\mathfrak{M}_{\mathbb{S}^{2d-1}}$ \cite{Palsson,JillPipher}, and very recently Borges and Foster \cite{Bo_Fo2023} have obtained almost sharp results including some endpoint estimates.
In this paper, we give a different proof of the same $(L^{p_{1}}\times L^{p_{2}}\rightarrow L^{p})$-estimate for $\mathfrak{M}_{\mathbb{S}^{2d-1}}$.

There is another important bi-(sub)linear maximal function
	$$\mathfrak{M}^{*}_{\mathbb{S}^{2d-1}}(f_{1},f_{2})(x):=\sup_{t>0}| \textsl{A}_{\mathbb{S}^{2d-1}}^{t}(f_{1},f_{2})(x)|,$$
which is known as bilinear spherical maximal function.
Study of this operator was generated in \cite{Grafakos1}. 
Later, in \cite{JeongLee} Jeong and Lee proved almost complete $L^{p_{1}}\times L^{p_{2}}\rightarrow L^{p}$ estimates for $1/p_1+1/p_2=1/p$, $p_{1},p_{2}>1 $ and $p>\frac{d}{2d-1}$ when $d\geq 2$. 
This result is extended to $d=1$ by Chirst and Zhou \cite{Christ_Zhou2022}.
It would be interesting to study $L^{p_1}\times L^{p_2}\rightarrow L^p$ boundedness of $\mathfrak{M}^{*}_{\Sigma}$, where $\Sigma$ is a compact smooth hypersurface with $\kappa$ non-vanishing principal curvatures $(\kappa\leq 2d-1)$. For some specific hypersurfaces, the optimal (except few border line cases) $L^{p_1}\times L^{p_2}\rightarrow L^p$ boundedness is known \cite{LeeShuin_2023}. 

For general hypersurface with non-vanishing Gaussian curvature, only $L^2(\R^d)\times L^2(\R^d) \to L^1(\R^d)$ estimate is known for $d\geq4$ \cite{GHH_2021}.
It would be interesting to study $L^{p_1}\times L^{p_2} \to L^p$ estimates of such full maximal averages for $p\leq 1$ in all dimensions and their multilinear analogues.
However, multilinear estimates for $m$-linear full maximal operators with $m\geq3$ have not been pursued, while $L^2\times\cdots\times L^2 \to L^{2/m}$ bounds for lacunary maximal operators are studied by Grafakos, He, Honz\'ik, and Park \cite{GHHP_2022}.
In this paper, we focus on $L^{p_1}\times\cdots\times L^{p_m} \to L^p$ bounds for the lacunary maximal functions for $1/p = 1/p_1+\cdots+1/p_m$ and $p<2/m$.
It would be our future goal to study $m$-linear estimates for the full maximal functions for $m\geq3$.

We first state $L^{1}$-improving and quasi-Banach estimates of the $m$-linear averaging operators $\mathcal{A}^{\Theta}_{\mathcal{S}}$ and $\textsl{A}_{\Sigma}$.
Note that the following two propositions are derived by simple Fourier analysis and multilinear interpolation, and we will give a proof of the propositions for self-containedness.
\begin{proposition}\label{prop-qbanach esti}
Let $\mathcal{A}_{\mathcal{S}}^\Theta(\rF)$ be given in \eqref{defn-mlao} and $\mathcal{S}$ be a compact smooth hypersurface contained in $\mathbb{B}^{d}(0,1)$ with $\kappa\leq d-1$ nonvanishing principal curvatures.  
Let $\Theta = \{\Theta_j\}_{j=1}^m$ be a family of mutually linearly independent rotaion matrices.
Let also $\mathcal{V}^{ij}_{\kappa}=\{z = (z_1, \dots, z_m) \in [0,1]^m : z_i = z_j = \frac{\kappa+1}{\kappa+2}, z_l = 0, l\not=i,j \}$	 and $\mathrm{conv}(\mathcal{V}_{\kappa})$ be its convex hull.
Then for $(\frac{1}{p_1} , \dots, \frac{1}{p_m})\in \mathrm{conv}(\mathcal{V}_{\kappa})$ we have the following inequalities:
%
 \begin{align*}
  \| \mathcal{A}_{\mathcal{S}}^\Theta(\rF) \|_{L^p(\R^{d})} &\lesssim  \prod_{j=1}^m \|f_j\|_{L^{p_j}(\R^{d})},
  \end{align*}
whenever $1\leq \frac{1}{p} \leq  \frac{2(\kappa+1)}{\kappa+2}=\sum_{j=1}^m\frac{1}{p_j} $.
\end{proposition}

\begin{proposition}\label{prop_qbesti_cM}
  Let $d\geq2$ and $\textsl{A}_{\Sigma}(\rF)$ be an average given by \eqref{defn-mlao-cM} over a compact smooth hypersurface $\Sigma$ with $\kappa$ nonvanishing principal curvatures with $(m-1)d < \kappa \leq md-1$.
  Then for $1\leq p_j \leq 2$, $j=1,2, \dots, m$ and $\frac{m+1}{2} \leq \sum_{j=1}^m \frac{1}{p_j} < \frac{2d+\kappa}{2d}$, the following $L^1$-improving estimates hold:
  \begin{align}\label{ineq_L1_impr}
  \| \textsl{A}_{\Sigma}(\rF) \|_{L^1(\R^{d})} \lesssim  \prod_{j=1}^m \|f_j\|_{L^{p_j}(\R^{d})}.
  \end{align}
  Moreover, we have for $\frac{m+1}{2} \leq \frac{1}{p} = \sum_{j=1}^m \frac{1}{p_j} < \frac{2d+\kappa}{2d}$
    \begin{align}\label{ineq_qba_cM}
    \| \textsl{A}_{\Sigma}(\rF) \|_{L^p(\R^{d})} \lesssim  \prod_{j=1}^m \|f_j\|_{L^{p_j}(\R^{d})}.
    \end{align}
  Let $1\leq p, p_1, \dots, p_m<\infty$ and $1/p=1/p_1+\cdots+1/p_m$.
  Then for $f_1, \dots, f_m$ with  $\supp(\widehat{f_j}) \subset \mathbb{A}_{n_j} := \{\xi_j \in \R^d : 2^{n_j-1} \leq |\xi_j| \leq 2^{n_j+1}\}$, $n_j\in \mathbb{Z}$, $j=1, \dots, m$, 
  we have
  \begin{align}\label{ineq_reg_cM}
		\left\| \textsl{A}_{\Sigma}(\mathrm{F})\right\|_{L^p(\R^d)} \lesssim 2^{-\delta|\mathbf{n}|} \prod_{j=1}^m \|f_j\|_{L^{p_j}(\R^d)},
	\end{align}
	where $\delta = \delta(p, \kappa , m , d)>0$ and $|n| = \sqrt{\sum_{j=1}^m n_j^2}$.
\end{proposition}

When $p>1$, one can obtain different $L^p$-improving estimates for $\cA_\cS^\Theta$ under specific choice of $\{\Theta_j\}$ and $\mathcal{S}$.
In this case, we do not need any curvature condition on $\cS$ and only the dimension of surfaces matters.
Let $\mathcal{S}^k$ be a $k$-dimensional $C^2$ surface in $\R^d$.
We choose mutually linearly independent $\{\Theta_j\}$. 
Moreover, we assume that for any choice of $\{j_i\}_{i=1}^\ell$ with $2\leq \ell \leq k+1\leq m$, the family $\{\Theta_j\}$ satisfies
\begin{align}
	\dim\Big( \text{span}_{1\leq i\leq \ell} \big( \{\Theta_{j_i}(y',0) \in \R^d : y'\in\R^k\} \big) \Big)&\geq \min\{k-1+\ell, d\},\label{ineq-240115 1454}\\
	\dim\Big(\bigcap_{i=1}^\ell \{\Theta_{j_i}(y',0) \in \R^d : y'\in\R^k\}  \Big) &\leq k+1-\ell.\label{ineq-230106 1631}
\end{align}
The assumption \eqref{ineq-230106 1631} yields that dimension of intersection of any subset $\{\Theta_{j_i}\}_{i=1}^{k+1}$ of $\{\Theta_j\}_{j=1}^m$ equals to zero.
The following theorem is one of our main results:
\begin{theorem}\label{thm-lp improving}
	Let $m\geq d\geq 2$ and $\mathcal{S}^k$ be a $k$-dimensional $C^2$ surface in $\mathbb{B}^d(0,1)$.
	Suppose that $\{\Theta_j\}$ satisfies \eqref{ineq-240115 1454} and \eqref{ineq-230106 1631},
	and $k$ is given such that
	\begin{align}
		\frac{m-d+k}{m} &\geq \frac{d-k-1}{d}k,\label{high codim}\\
		\frac{m-1}{m} &\geq \frac{(d-k)k}{d}. \label{hypersurface} 
	\end{align}
	Then $\mathcal{A}_{\mathcal{S}^k}^\Theta$ is of strong type $(m, \dots, m, \frac{d}{d-k})$.
  That is, we have
	\begin{align}
		\| \mathcal{A}_{\mathcal{S}^k}^\Theta (\rF) \|_{L^{\frac{d}{d-k}}(\R^d)} \lesssim \prod_{j=1}^m \|f_j \|_{L^{m}(\R^d)}.
	\end{align}
\end{theorem}
In our proof of Theorem~\ref{thm-lp improving}, we mainly use the nonlinear Brascamp-Lieb inequality proved in \cite{BBBCF2020}.
We give details on the inequality and the proof of Theorem~\ref{thm-lp improving} in Section~\ref{sec-nlbl}.

In Theorem~\ref{thm-lp improving}, One can use $m\geq d$ to check that \eqref{high codim} and \eqref{hypersurface} are equivalent when $d= 2k+1$.
Precisely, \eqref{high codim} implies \eqref{hypersurface} when $d\geq 2k+1$, and \eqref{hypersurface} implies \eqref{high codim} when $d\leq 2k+1$.
Moreover, if we assume $k=d-1$, then we only need \eqref{ineq-230106 1631} to guarantee the following result:
\begin{corollary}\label{cor-hypersurface}
  Let $m\geq d\geq 2$, $\mathcal{S}^{d-1}$ be a $C^2$ hypersurface, and $\{\Theta_j\}$ be chosen to be mutually linearly independent and satisfy \eqref{ineq-230106 1631}. Then , $\mathcal{A}_{\mathcal{S}^{d-1}}^\Theta$ is of strong-type $(m, \dots, m, d)$.
\end{corollary}
One can find similar results in \cite[Theorem 1.2]{Iosevich}, which yields restricted strong-type $(m, \dots, m, m)$ and $\Big(m\frac{d+1}{d}, \dots, m\frac{d+1}{d}, d+1\Big)$ estimates for $\cA_{\cS^{d-1}}^\Theta$ when $\mathcal{S}^{d-1}$ is a sphere.
Note that the authors \cite{Iosevich} considers $m\leq d$ cases with linearly independent $\{\Theta_j\}$, so it cannot be directly compared to Corollary~\ref{cor-hypersurface} in which $m\geq d$ and \eqref{ineq-230106 1631} are considered.
When $m=d$, however, Corollary~\ref{cor-hypersurface} with $\mathcal{S}^{d-1} = \mathbb{S}^{d-1}$gives strong-type $(m, \dots, m, m)$ estimates.

To study further how \cite[Theorem 1.2]{Iosevich} and Corollary~\ref{cor-hypersurface} are related, we introduce a quantity $\mathfrak{D}$ which is  given for each $(p_1, \dots, p_m, p)$-estimate by
$$
	\mathfrak{D}(p_1, \dots, p_m;p) := \Big(\frac{1}{p_1} + \cdots +\frac{1}{p_m}\Big) - \frac{1}{p}.
$$
One can measure extent of $L^p$-improving by means of the difference $\mathfrak{D}$.
Then we have
$$
	\mathfrak{D}(m, \dots, m; m) = \frac{m-1}{m},\quad \mathfrak{D}\Big(m\frac{d+1}{d}, \dots, m\frac{d+1}{d}; d+1\Big) = \frac{d-1}{d+1},
$$ 
where $m\leq d$.
On the other hand, Corollary~\ref{cor-hypersurface} yields 
$$
	\mathfrak{D}(m, \dots, m; d) = \frac{d-1}{d},\quad m\geq d.
$$
Thus Corollary~\ref{cor-hypersurface} yields wider range of $L^p$-improving than $(m(d+1)/d, \dots, m(d+1)/d, d+1)$-estimate of \cite[Theorem 1.2]{Iosevich} under certain choice of $\{\Theta_j\}$.
We also note that the difference $\frac{d-1}{d+1}$ is the best possible for linear spherical averages, since $\cA_{\mathbb{S}^{d-1}}$ satisfies $L^{\frac{d+1}{d}}(\R^d) \to L^{d}(\R^d)$ boundedness.
Even for $L^1$-improving estimates in Proposition~\ref{prop-qbanach esti}, we obtain $\mathfrak{D}(p_1, \dots, p_m; 1) = \frac{d-1}{d+1}$.
Hence one can say that the number $\frac{d-1}{d}$ only occurs for multilinear averaging operators with certain transversality of $\{\Theta_j\}$.
Moreover, we only assume a surface $\cS$ is of class $C^2$ without any curvature condition,
and it would be very interesting to study boundedness of maximal operators associated with $\mathcal{A}_{\mathcal{S}^k}^\Theta$.

\hfill

By making use of the quasi-Banach space estimates, Propositions~\ref{prop-qbanach esti} and \ref{prop_qbesti_cM} together with Sobolev regularity estimates, we obtain multilinear estimates for lacunary maximal operators $\cM_\cS^\Theta$ and $\mathfrak{M}_\Sigma$.
\begin{theorem}\label{thm-lac}
	Let $1\leq p_i^\circ \leq\infty$, $\sum^m_{i=1}\frac{1}{p_i^\circ}=\frac{1}{p^\circ}$ with $p^\circ\geq1$ for $d\geq2$. 
	Suppose that $\mathcal{A}_\mathcal{S}^\Theta$ satisfies the following Sobolev regularity estimates:
  \begin{align}\label{ineq_reg_cM'}
		\left\| \mathcal{A}_{\mathcal{S}}^\Theta(\mathrm{F})\right\|_{L^{p^\circ}(\R^d)} \lesssim 2^{-\varepsilon|\mathbf{n}|} \prod_{j=1}^m \|f_j\|_{L^{p_j^\circ}(\R^d)},
	\end{align}
	where $f_1, \dots, f_m$ with  $\supp(\widehat{f_j}) \subset \mathbb{A}_{n_j} := \{\xi_j \in \R^d : 2^{n_j-1} \leq |\xi_j| \leq 2^{n_j+1}\}$, $j=1, \dots, m$, and $\varepsilon = \varepsilon(p, \kappa , m , d)>0$.
  Then the lacunary maximal function $\mathcal{M}^{\Theta}_{\mathcal{S}}$ maps $L^{p_1}(\R^d)\times\cdots\times L^{p_m}(\R^d)\rightarrow L^{p}(\R^d)$ for $(\frac{1}{p_1}, \dots, \frac{1}{p_m})\in  \mathrm{conv}(\mathcal{V}_{\kappa}^\circ) \cup \{(0,\dots, 0)\}$ and $1/p = 1/p_1+\cdots+1/p_m$, 
  where $\mathrm{conv}(\mathcal{V}_{\kappa}^\circ)$ denotes an interior of the convex hull of $\mathrm{conv}(\mathcal{V}_{\kappa})$ and the origin.
  In particular, if one considers a lacunary maximal operator associated with $\mathbb{S}^{d-1}$, then the range of $p$ becomes $p>\frac{d+1}{2d}$.
\end{theorem}

Observe that the multilinear averaging operator \eqref{defn-mlao} is an analogous multilinear averaging operator to the bilinear operator $\textsl{B}_{\theta}$ considered by Greenleaf $et.$ $al.$ \cite{Greenleaf1}. 
$$
	B_\theta(f,g)(x) = \int_{\mathbb{S}^1} f(x-y) g(x-\theta y) ~\mathrm{d}\sigma(y),
$$
where $\theta$ denotes a counter-clockwise rotation. 
Therefore, Theorem~\ref{thm-lac} (when $m=2$) yields boundedness of the lacunary maximal function corresponding to the averaging operator $\textsl{B}_{\theta}$ under the assumption on the Sobolev regularity estimates \eqref{ineq_reg_cM'}.
Thus one only need to show \eqref{ineq_reg_cM'}, but it is not accomplished in this paper.

On the other hand, one can actually obtain Sobolev regularity estimates for $\textsl{A}_\Sigma$, which is \eqref{ineq_reg_cM} of Proposition~\ref{prop_qbesti_cM}.
Thus, another main result of this paper is the following lacunary maximal estimates for $\textsl{A}_\Sigma$:
\begin{theorem}\label{thm-lac-cM}
	Let $\frac{m+1}{2}\leq \frac{1}{p} = \sum_{j=1}^m \frac{1}{p_j}<\frac{2d+\kappa}{2d}$ for $1\leq p_j\leq2$ and $\kappa>(m-1)d$. 
	Then the lacunary maximal operator $ \mathfrak{M}_{\Sigma}$ maps $L^{p_1}(\R^d)\times\cdots\times L^{p_m}(\R^d)\rightarrow L^{p}(\R^d)$.
\end{theorem}
$(L^{p_1}\times\cdots\times L^{p_m} \to L^p)$-estimates of Theorem~\ref{thm-lac-cM} is easily extended to $1\leq p_j \leq\infty$ via multilinear interpolation, since $ \mathfrak{M}_{\Sigma}$ is bounded from $L^\infty \times \cdots \times L^\infty$ to $L^\infty$.

\begin{remark}\label{rem_mainthm}
    What we will prove in Sections~\ref{sec_thm_lac} and \ref{sec-qba n smthng}  is that multi-linear estimates of lacunary maximal operators will be derived from $L^1$-improving estimates and Sobolev regularity estimates of corresponding averaging operators.
    Precisely, if one obtains $L^{p_1^\circ} \times \cdots \times L^{p_m^\circ}\to L^1$ estimates of averaging operators with $\sum_{j=1}^m 1/p_j^\circ >1$, then we obtain $L^{p_1^\circ} \times \cdots \times L^{p_m^\circ} \to L^{p^\circ, \infty}$ estimates of the lacunary maximal operators for $\sum_{j=1}^m 1/p_j^\circ =1/p^\circ$ together with certain polynomial growth, which is Lemma~\ref{lem-qba esti}.
    The polynomial growth of Lemma~\ref{lem-qba esti} will be handled by interpolation with an exponential decay estimates of Lemma~\ref{lem-smoothing}, which is originated by the Sobolev regularity estimates of averaging operators.
    In result, we obtain $L^{p_1}\times \cdots \times L^{p_m} \to L^p$ estimates for $\sum_{j=1}^m 1/p_j =1/p$ with $p_1, \dots, p_m\geq1$ and $p>p^\circ$. 
\end{remark}

As a simple application of Remark~\ref{rem_mainthm}, we obtain the following result:
\begin{remark}\label{bilinearspherical}
	Theorem \ref{thm-lac-cM} also yields the following boundedness of the bilinear lacunary spherical maximal function $ \mathfrak{M}_{\mathbb{S}^{2d-1}}$.
	 Let $d\geq1$, $1<p_{1},p_2\leq\infty$ and $1/p_{1}+1/p_{2}=1/p$. Then 
	 \begin{align}\label{ineq_bi_lac_sph}
	 	\Vert  \mathfrak{M}_{\mathbb{S}^{2d-1}}(f_1,f_2)\Vert_{L^{p}}\lesssim \Vert f_1\Vert_{L^{p_1}}\Vert f_2\Vert_{L^{p_2}}.
	\end{align}
	 Note that we make use of $L^1\times L^1\to L^{1/2}$ estimates of $\textsl{A}_{\mathbb{S}^{2d-1}}^1$ given by \cite{Iosevich} and machinery of Section~\ref{sec_thm_lac} to obtain \eqref{ineq_bi_lac_sph} for $p>1/2$.
	 This estimate is already given in \cite{Bo_Fo2023} and we give a different proof at the end of this paper. 
\end{remark}

\begin{remark}
  It is known in \cite{GHHP_2022} that $\mathfrak{M}_{\Sigma}$ satisfies $(L^2\times \cdots\times L^2 \to L^{2/m})$-estimates for certain $\kappa$.
  One can check that even for the worst indices, our Theorem \ref{thm-lac-cM} is better than the $(L^2\times\cdots\times L^2 \to L^{2/m})$-estimates in the sense that Theorem~\ref{thm-lac-cM} holds for $L^p$ spaces with lower indices since $\frac{2}{m} > \frac{2}{m+1} > \frac{2d}{2d+\kappa}$.
  When $\kappa\leq(m-1)d$, we do not know anything yet.
  \end{remark}

\section*{Notations and definitions}

$\bullet$ For a cube $Q$ or a ball $B$ in $\R^d$, we define $CQ$ and $CB$ whose sidelength and radius are $C$ times those of $Q$ and $B$ with the same centers, respectively.
For a measureable set $E$, we say $\meas(E)$ as a measure of $E$.

$\bullet$
Choose a Schwartz class function $\phi$ such that $\supp(\widehat{\phi})\subset B(0,2)$ and $\widehat{\phi}(\xi)=1$ for $\xi\in B(0,1)$. 
Also consider $\widehat{\psi}(\xi)=\widehat{\phi}(\xi)-\widehat{\phi}(2\xi)$ so that $\supp(\wh{\psi}) \subset \{2^{-1}<|\xi|<2\}$.
By symbols $\wh{\phi}_\ell(\xi) = \wh{\phi}(2^{-\ell}\xi)$ and $\wh{\psi}_\ell(\xi) = \wh{\psi}(2^{-\ell}\xi)$ we define frequency projection operators.
 \begin{align}
  \widehat{P_{<\ell}f}(\xi)&=\widehat{f}(\xi)\widehat{\phi}_\ell(\xi),\label{proj<ell}\\
  \widehat{P_{\ell}f}(\xi)&=\widehat{f}(\xi)\widehat{\psi}_\ell(\xi).\label{proj_ell}.
 \end{align}

%
%
%

\section{Proofs of Propositions~\ref{prop-qbanach esti} and \ref{prop_qbesti_cM}}\label{section 2}

\subsection{Proof of Proposition~\ref{prop-qbanach esti}}\label{pf_qbanach esti}
The proof of Proposition~\ref{prop-qbanach esti} follows from the following lemma and 
 a standard technique from \cite{GrKa_2001, Iosevich}.
\begin{lemma}\label{lem-L1 improving1}
  $\mathcal{A}_{\mathcal{S}}^\Theta$ is bounded from $L^{p_1}(\R^d) \times \cdots\times L^{p_m}(\R^d)$ to $L^1(\R^d)$ for $(\frac{1}{p_1} , \dots, \frac{1}{p_m}) \in \mathrm{conv}(\mathcal{V}_{\kappa})$.
  In particular, if $\kappa = d-1$, then one example of $\mathcal{A}_{\mathcal{S}}^\Theta$ is the spherical averaging operator $\mathcal{A}_{\mathbb{S}^{d-1}}^\Theta$.
\end{lemma}

  Let $p = \frac{k+2}{2(k+1)}$ and $(\frac{1}{p_1} , \dots, \frac{1}{p_m}) \in \mathrm{conv}(\mathcal{V}_{\kappa})$.
  We begin with
  \begin{align*}
    \| \mathcal{A}_{\mathcal{S}}^\Theta(\rF) \|_{L^p(\R^{d})}^p 
    = \int_{\R^{d}} \Big| \int_{\mathcal{S}} \prod_{j=1}^m f_j (x + \Theta_j y)  ~\mathrm{d}\sigma(y) \Big|^p ~\mathrm{d}x.
  \end{align*}
  Decompose $\R^d$ into countable union of unit cubes ${Q}_{\bf{n}} = {\bf{n}} + [0,1)^d$, ${\bf{n}} \in \mathbb{Z}^d$.
  Together with compactness of $\mathcal{S}$, we have
  \begin{align}\label{ineq-221121 1946}
    \| \mathcal{A}_{\mathcal{S}}^\Theta(\rF)\|_{L^p(\R^d)}^p 
    = \sum_{{\bf{n}}\in\mathbb{Z}^{d}} \int_{{Q}_{\bf{n}}} \Big| \int_{\mathcal{S}} \prod_{j=1}^m f_j (x + \Theta_j y)  ~\mathrm{d}\sigma(y) \Big|^p ~\mathrm{d}x.
  \end{align}
  Now we apply H\"older's inequality to obtain
  \begin{equation}\label{ineq-221121 1947}
  \begin{aligned}
    &\int_{{Q}_{\bf{n}}} \Big| \int_{\mathcal{S}} \prod_{j=1}^m f_j (x + \Theta_j y) ~\mathrm{d}\sigma(y) \Big|^p ~\mathrm{d}x\\
    &\lesssim \Bigl( \int_{Q_{\bf{n}}}  \int_{\mathcal{S}} \Big| \prod_{j=1}^m f_j (x + \Theta_j y) \Big| ~\mathrm{d}\sigma(y)  ~\mathrm{d}x\Bigr)^{p}.
  \end{aligned}
  \end{equation}
  Since $x\in {Q}_{\bf{n}}$ and $y \in \supp(\mathcal{S})$, we have the following equality
  \begin{align}\label{ineq-221121 1932}
    f_j(x + \Theta_j y) = ( f_j \mathbbm{1}_{\widetilde{Q}_{\bf{n}}} )(x + \Theta_j y),
  \end{align}
  where $\widetilde{Q}$ denotes a cube whose sidelength is 3 times  that of ${Q}$ with the same center.
  With help of \eqref{ineq-221121 1932} and Lemma~\ref{lem-L1 improving1}, we have
  \begin{equation}\label{ineq-221121 1948}
  \begin{aligned}
    &\Bigl( \int_{{Q}_{\bf{n}}}  \int_{\mathcal{S}} \Big| \prod_{j=1}^m f_j (x + \Theta_j y) \Big|  ~\mathrm{d}\sigma(y)  ~\mathrm{d}x\Bigr)^{p}\\
    &= \Bigl( \int_{{Q}_{\bf{n}}}  \int_{\mathcal{S}} \Big| \prod_{j=1}^m (f_j \mathbbm{1}_{\widetilde{Q}_{\bf{n}}}) (x + \Theta_j y) \Big| ~\mathrm{d}\sigma(y)  ~\mathrm{d}x\Bigr)^{p}\\
    &\lesssim  \Big(\prod_{j=1}^m \|  (f_j \mathbbm{1}_{\widetilde{Q}_{\bf{n}}}) \|_{L^{p_j}(\R^{d})} \Big)^{p}
    =\prod_{j=1}^m \|  f_j \mathbbm{1}_{\widetilde{Q}_{\bf{n}}} \|_{L^{p_j}(\R^{d})}^p
  \end{aligned}
  \end{equation}
  whenever $(\frac{1}{p_1}, \dots, \frac{1}{p_m})$ is in $\mathrm{conv}(\mathcal{V}_{\kappa})$.

  By \eqref{ineq-221121 1946}, \eqref{ineq-221121 1947}, and \eqref{ineq-221121 1948}, we have
  \begin{align}\label{ineq-221121 1949}
    \|\mathcal{A}_{\mathcal{S}}^\Theta(\rF) \|_{L^p(\R^{d})}^p 
    \lesssim \sum_{{\bf{n}}\in\mathbb{Z}^{d}} \prod_{j=1}^m \|  f_j \mathbbm{1}_{\widetilde{Q}_{\bf{n}}} \|_{L^{p_j}(\R^{d})}^p.
  \end{align}
  We make use of H\"older's inequality on \eqref{ineq-221121 1949} to obtain
  \begin{align}
    \| \mathcal{A}_{\mathcal{S}}^\Theta(\rF)\|_{L^p(\R^{d})}^p 
    \lesssim \prod_{j=1}^m \Bigl( \sum_{{\bf{n}} \in \mathbb{Z}^{d}} \| f_j \mathbbm{1}_{\widetilde{Q}_{\bf{n}}} \|_{L^{p_j}(\R^{d})}^{p_j} \Bigr)^{\frac{p}{p_j}}.
  \end{align}
  Note that $\{\widetilde{Q}_{\bf{n}}\}_{{\bf{n}}\in\mathbb{Z}^{d}}$ is a finitely overlapping cover of $\R^{d}$.
  Therefore we have
  \begin{align}
    \| \mathcal{A}_{\mathcal{S}}^\Theta(\rF) \|_{L^p(\R^{d})} \lesssim  \prod_{j=1}^m \|f_j\|_{L^{p_j}(\R^{d})},
  \end{align}
where $\frac{1}{p} = \sum_{j=1}^m\frac{1}{p_j}$, and $(\frac{1}{p_1} , \dots, \frac{1}{p_m}) \in \mathrm{conv}(\mathcal{V}_{\kappa})$.

Thus, by showing Lemma~\ref{lem-L1 improving1}, we complete the proof of Proposition~\ref{prop-qbanach esti}.
\begin{proof}[Proof of Lemma~\ref{lem-L1 improving1}]
  Let $\mathfrak{a}$ be a symbol satisfying $|\mathfrak{a}(\xi)| \lesssim (1+|\xi|)^{-\rho}$ for some $\rho>0$.
  Then for $m=1$, it is well-known \cite{Stichartz, Littman} that 
  $T_\mathfrak{a}(f) = (\mathfrak{a} \widehat{f}\,)\check{\,}$ is bounded from $L^p(\R^d)$ to $L^{p'}(\R^d)$ 
  for $\frac{1}{p}+\frac{1}{p'}=1$, $p\in[1,2]$, and $\frac{1}{p} - \frac{1}{2}\leq \frac{1}{2}(\frac{\rho}{\rho+1})$.
  Let $\mathcal{S}$ be a hypersurface with $\kappa$ nonvanishing principal curvatures.
  For bilinear case $m=2$, by change of variables we have
  \begin{align*}
    \| \mathcal{A}_{\mathcal{S}}^\Theta(f,g) \|_1 
    &\leq \int_{\R^d} \int_{\mathcal{S}}\Big|f(x+\Theta_1y) g(x+\Theta_2 y)\Big| ~\mathrm{d}\sigma(y)~\mathrm{d}x\\
    &= \int_{\R^d} |f(x)| \int_{\mathcal{S}}\Big|g(x+(\Theta_2-\Theta_1) y)\Big| ~\mathrm{d}\sigma(y)~\mathrm{d}x\\
    &\leq \|f\|_p \|g\|_p,
  \end{align*}
  where the last inequality follows from H\"older's inequality and $\frac{1}{p} = \frac{\kappa+1}{\kappa+2}$.
  Thus for $m$-linear case, it follows that
  \begin{align*}
    \| \mathcal{A}_{\mathcal{S}}^\Theta(\rF) \|_1
    &\leq  \int_{\R^d} \int_{\mathcal{S}}\Big|\prod_{j=1}^m f_j(x+\Theta_jy) \Big| ~\mathrm{d}\sigma(y)~\mathrm{d}x\\
    &\leq \int_{\R^d} \int_{\mathcal{S}}\Big|f_{1}(x+\Theta_{1}y) f_{2}(x+\Theta_{2}y) \Big| ~\mathrm{d}\sigma(y)~\mathrm{d}x \times \prod_{3\leq j\leq m} \|f_{j}\|_\infty\\
    &\leq \| f_{1}\|_{p} \|f_{2}\|_{p}  \prod_{3\leq j\leq m} \|f_{j}\|_\infty,
  \end{align*}
  where $\frac{1}{p}= \frac{\kappa+1}{\kappa+2}$.
  Similarly, interchanging  the role of the functions and invoking multilinear interpolation we get the desired estimate.
\end{proof}

\subsection{Proof of Proposition~\ref{prop_qbesti_cM}}

\subsubsection{$L^1$- improving estimates \eqref{ineq_L1_impr}}
By translation $x \to x+y_m$, we reduce the $L^1$-norm of $\textsl{A}_{\Sigma}$ into $L^\infty$-norm of the following $(m-1)$-linear operator:
\begin{align}\label{m-1_reduction}
	\int_{\Sigma} \prod_{j=1}^{m-1} |f_j(x+y_m - y_j)|~\mathrm{d}\sigma_\Sigma(y).
\end{align}
By using the Fourier transform, we rewrite \eqref{m-1_reduction} as
\begin{equation}\label{ineq_230228_1810}
\begin{aligned}
	&\int_{\R^{(m-1)d}} e^{2\pi i x\cdot(\xi_1 +\cdots+\xi_{m-1})} \widehat{\mathrm{d}\sigma_\Sigma}(\xi', -\xi_1 - \cdots - \xi_{m-1}) \prod_{j=1}^{m-1} \widehat{|f_j|}(\xi_j) ~\mathrm{d}\xi',
\end{aligned}
\end{equation}
where $\xi' = (\xi_1, \dots, \xi_{m-1}) \in\R^{(m-1)d}$.
 
Since the hypersurface $\Sigma$ has $\kappa$ nonvanishing principal curvatures, therefore using the result of Littman \cite{Littmanprincipalcurvature} we get $|\widehat{\mathrm{d}\sigma_\Sigma}(\xi)| \lesssim (1 + |\xi|)^{-\kappa/2}$ for $\xi\in\R^{md}$. This implies that the symbol of \eqref{ineq_230228_1810} satisfies
$$
	\left|  \widehat{\mathrm{d}\sigma_\Sigma}(\xi', -\xi_1 - \cdots - \xi_{m-1})  \right| \lesssim (1+|\xi'|)^{-\kappa/2}.
$$
By applying H\"older's inequality on \eqref{ineq_230228_1810}, it is bounded above by
 \begin{equation*}
	\Big(\int_{\R^{(m-1)d}}  \prod_{j=1}^{m-1} |\widehat{|f_j|}(\xi_j)|^{p'} ~\mathrm{d}\xi'\Big)^{1/p'}\times\Big( \int_{\R^{(m-1)d}}   (1+|\xi'|)^{-\frac{\kappa p}2} ~\mathrm{d}\xi' \Big)^{1/p}
\end{equation*}
and the last term is finite if $ p > \frac{2d(m-1)}{\kappa}$.
Thus, for $ \frac{2d(m-1)}{\kappa}< p < 2$ we have 
\begin{equation}\label{ineq_230228_1820}
\begin{aligned}
	\Big| \int_\Sigma \prod_{j=1}^{m-1} f(x+y_m - y_j)~\mathrm{d}\sigma_{\Sigma}(y) \Big|
	\lesssim \prod_{j=1}^{m-1} \|\widehat{|f_j|}\|_{L^{p'}(\R^d)}.
\end{aligned}
\end{equation}
Together with $L^1$-norm of $f_m$,  for $\frac{2d(m-1)}{\kappa} < p\leq 2$ we have
\begin{align}\label{ineq_230228_1823}
	\Big\|\textsl{A}_{\Sigma}(\mathrm{F})\Big\|_{L^1(\R^d)} \lesssim \prod_{j=1}^{m-1} \|f_j\|_{L^{p}(\R^d)} \times \|f_m\|_{L^1(\R^d)}.
\end{align}
Symmetry of estimates \eqref{ineq_230228_1823} and multilinear interpolation yield that
\begin{align}\label{Mavg_improving}
	\Big\|\textsl{A}_{\Sigma}(\mathrm{F})\Big\|_{L^1(\R^d)} \lesssim \prod_{j=1}^{m} \|f_j\|_{L^{p_j}(\R^d)},
\end{align}
where $\frac{m+1}{2} \leq \sum_{j=1}^m \frac{1}{p_j} < \frac{2d+\kappa}{2d}$ and $1\leq p_j \leq 2$.

\subsubsection{Quasi-Banach space estimates \eqref{ineq_qba_cM}}
Since we obtain $L^1$-improving estimates for $\textsl{A}_{\Sigma}$, one can apply the argument of Subsection~\ref{pf_qbanach esti} to show that $\textsl{A}_{\Sigma}$ satisfies a H\"older-type multilinear estimates on $L^p(\R^d)$ for $\frac{1}{p} = \sum_{j=1}^m \frac{1}{p_j}$ with $p_j$ in \eqref{Mavg_improving}.
That is, we have
\begin{align}\label{Mavg_qbanach}
	\Big\|\textsl{A}_{\Sigma}(\mathrm{F})\Big\|_{L^p(\R^d)} \lesssim \prod_{j=1}^{m} \|f_j\|_{L^{p_j}(\R^d)},
\end{align}
where $\frac{m+1}{2} \leq \sum_{j=1}^m \frac{1}{p_j} < \frac{2d+\kappa}{2d}$ and $1\leq p_j \leq 2$.
This proves the quasi-Banach space estimates.

\subsubsection{smoothing estimates \eqref{ineq_reg_cM}}
For the Sobolev regularity estimates, note that $\textsl{A}_{\Sigma}$ is written in terms of Fourier multipliers.
	\begin{align}\label{ineq_230228_1250}
		 \textsl{A}_{\Sigma}(\mathrm{F})(x) 
		= \int_{\R^{md}} e^{2\pi i x\cdot(\xi_1+\cdots+\xi_m)} \widehat{\mathrm{d}\sigma_\Sigma}(\vec{\xi}\,) \prod_{j=1}^m \widehat{f}_j(\xi_j)~\mathrm{d}\vec{\xi}.
	\end{align}
	Moreover, we consider $f_1, \dots, f_m$ whose Fourier transforms are supported in $\{\xi \in \R^d : 2^{n_j-1}\leq |\xi|\leq 2^{n_j+1}\}$ with positive integers $n_j$, $j=1,\dots, m$, respectively.
	Since $\widehat{\mathrm{d}\sigma_\Sigma}$ satisfies the following limited decay condition:
	\begin{align}\label{limdecay}
		| \partial^\alpha \widehat{\mathrm{d}\sigma_\Sigma}(\vec{\xi}\,) | \lesssim (1+|\vec{\xi}\,|)^{-\kappa/2}\quad\text{for any multi-indices $\alpha$},
	\end{align}
	we are going to make use of one of main results of \cite{GHHP_2023} initial estiamtes.
	\begin{theorem}[Theorem 1.1, \cite{GHHP_2023}]\label{thm_initial}
		Let $m$ be a positive number with $m\geq2$ and $1<q<\frac{2m}{m-1}$.
		Set $M_q$ to be a positive integer satisfying 
		$$
			M_q>\frac{m(m-1)d}{2m-(m-1)q}.
		$$
		Suppose that $\mathfrak{m} \in L^q(\R^{md}) \cap C^{M_q}(\R^{md})$ with
		$$
			\| \partial^\alpha \mathfrak{m} \|_{L^\infty(\R^{md})}\leq D_0,\quad\text{for $|\alpha|\leq M_q$}.
		$$
		Then we have
		$$
			\|T_\mathfrak{m}(f_1, \dots, f_m)\|_{L^{2/m}(\R^d)}\lesssim D_0^{1- \frac{(m-1)q}{2m}} \| \mathfrak{m} \|_{L^q(\R^{md})}^{\frac{(m-1)q}{2m}} \prod_{j=1}^m\|f_j\|_{L^2(\R^d)}.
		$$
	\end{theorem}
	Note that $T_\mathfrak{m}(f_1, \dots, f_m)$ is a multilinear operator whose Fourier multiplier is $\mathfrak{m}$.
	Then by putting \eqref{limdecay} into Theorem~\ref{thm_initial}, we have
	\begin{align}
		\mathfrak{m}(\xi) &= \widehat{\mathrm{d}\sigma}(\xi)\prod_{j=1}^m\widehat{\psi}_{n_j}(\xi_j),\\
		D_0 &\simeq 1,\\
		\|\mathfrak{m}(\xi)\|_{L^q(\R^{md})} &\lesssim 2^{-|\mathbf{n}|\kappa/2} 2^{|\mathbf{n}|md/q}.
	\end{align}
	Since $q\in(1, \frac{2m}{m-1})$, for $f_1, \dots, f_m$ whose Fourier transforms are supported in $\mathbb{A}_{n_j} = \{\xi_j \in \R^d : 2^{n_j-1} \leq |\xi_j| \leq 2^{n_j+1}\}$ we have
	\begin{align}\label{multi_L2}
		\| A_\Sigma(\rF) \|_{L^{2/m}(\R^d)} \lesssim 2^{-|\mathbf{n}|(\frac{\kappa}{2} - \frac{(m-1)d}{2})} \prod_{j=1}^m \|f_j\|_{L^2(\R^d)}.
	\end{align}
	Note that we have trivial estimates for $1\leq p, p_1, \dots, p_m\leq \infty$ with $1/p = 1/p_1 +\cdots+1/p_m$, 
	\begin{align}\label{multi_trivial}
		\| A_\Sigma(\rF) \|_{L^{p}(\R^d)} \lesssim  \prod_{j=1}^m \|f_j\|_{L^{p_j}(\R^d)}.
	\end{align}
	By interpolating \eqref{multi_L2} and  \eqref{multi_trivial}, for any $1\leq p, p_1, \dots, p_m< \infty$ with $1/p = 1/p_1 +\cdots+1/p_m$
	there is an $\delta = \delta(p, \kappa, m, d)>0$ such that
	\begin{align*}
		\| A_\Sigma(\rF) \|_{L^{p}(\R^d)} \lesssim  2^{-\delta|\mathbf{n}|}\prod_{j=1}^m \|f_j\|_{L^{p_j}(\R^d)},
	\end{align*}
	when $\widehat{f_j}$ is supported in $\mathbb{A}_{n_j}$ for $j=1, \dots, m$. 
	This proves \eqref{ineq_reg_cM}.

\section{A Nonlinear Brascamp-Lieb inequality approach to $L^p$-improving estimates for $\cA_\cS^\Theta$}\label{sec-nlbl}

\subsection{Nonlinear Brascamp-Lieb inequality}

Let $f_j$ be nonnegative integrable funtions, $L_j : \R^{d} \to \R^{d_j}$ be linear surjections, and $c_j \in [0,1]$ for $j=1,\cdots, m$.
We also identify a finite dimensional Hilbert sapce $H$ and a Euclidean space $\R^n$, for instance we let $H = \R^d$ and $H_j = \R^{d_j}$.
Then we can consider the linear Brascamp-Lieb inequality 
\begin{align}\label{LBL}
  \int_{\R^d} \prod_{j=1}^m \Bigl(f_j (L_j x)\Bigr)^{c_j} \, ~\mathrm{d}x 
  \leq \mathrm{BL}(\mathbf{L}, \mathbf{c}) \prod_{j=1}^m \Bigl(\int_{\R^{d_j}} f_j(x_j) ~\mathrm{d}x_j \Bigr)^{c_j},
\end{align}
where $\mathbf{L} = (H, \{H_j\}_{1\leq j\leq m}, \{L_j\}_{1\leq j\leq m})$, $\mathbf{c} = (c_1, \dots, c_m)$, and $\mathrm{BL}(\mathbf{L}, \mathbf{c})$ is the smallest such constant.
Here, we call $(\mathbf{L}, \mathbf{c})$ a Brascamp-Lieb datum, and $\mathrm{BL}(\mathbf{L}, \mathbf{c})$ a Brascamp-Lieb constant. 
There have been studies on nonlinear generalizations of the Brascamp-Lieb inequality.
Bennett, Carbery, and Wright \cite{BCW2005} showed that \eqref{LBL} holds for $d_j = d-1$ and $c_j = \frac{1}{m-1}$ when $L_j$'s are smooth submersions supported in a sufficiently small neighborhood.
They also proved that $L_j$'s could be $C^3$ mappings under certain transversality conditions on the submersions. 
Later, Bennett and Bez \cite{BB2010} extended the results of \cite{BCW2005} to general $d_j$ and $C^{1,\beta}$ mappings.
Recently, Bennett, Bez, Buschenhenke, Cowling, and Flock \cite{BBBCF2020} proved the following nonlinear Brascamp-Lieb inequality.
\begin{theorem}\cite[Theorem 1.1]{BBBCF2020}\label{thm-NLBL}
  Let $(\mathbf{L}, \mathbf{c})$ be a Brascamp-Lieb datum.
  Suppose that $B_j : \R^d \to \R^{d_j}$ are $C^2$ submersions in a neighborhood of a point $x_0$ and $~\mathrm{d}B_j(x_0) = L_j$, $j=1,2,\cdots ,m$.
  Then for each $\varepsilon>0$ there exists a neighborhood $U$ of $x_0$ such that
  \begin{align}
    \int_U \prod_{j=1}^m \Bigl(f_j (B_j (x))\Bigr)^{c_j}\, ~\mathrm{d}x 
    \leq (1+\varepsilon) \mathrm{BL}(\mathbf{L}, \mathbf{c}) \prod_{j=1}^m \Bigl(\int_{\R^{d_j}} f_j(x_j) ~\mathrm{d}x_j\Bigr)^{c_j}.
  \end{align}
\end{theorem}
Although Theorem~\ref{thm-NLBL} is stated with $C^2$ submersions, the proof of \cite{BBBCF2020} guarantees that the theorem still holds if one takes $C^{1+\theta}$ submersions for any $\theta>0$.
It is known \cite{BCCT2008} that $\mathrm{BL}(\mathbf{L}, \mathbf{c})$ is finite if and only if the following conditions hold.
\begin{align}
  \dim(V) &\leq \sum_{j=1}^m c_j \dim(L_j V)\quad\text{for all subspaces $V$ of $\R^d$},\label{condi-transv}\\
  d &= \sum_{j=1}^m c_j d_j.\label{condi-scaling}
\end{align}
These conditions are called the transversality condition and the scaling condition, respectively.
We also present necessary conditions for finiteness of $\mathrm{BL}(\mathbf{L}, \mathbf{c})$:
\begin{align*}
  \bigcap_{j=1}^m \ker(L_j) = \{0\}, \quad \sum_{j=1}^m c_j \geq1.
\end{align*}
But, it is not simple to check \eqref{condi-transv} for a given Brascamp-Lieb datum.
The following lemma may be useful in such verification.
First, we say a proper subspace $V_c$ of $\R^d$ is a critical subspace if it satisfies 
$$
  \dim(V_c) = \sum_{j=1}^m c_j \dim(L_j V_c).
$$
For a given subspace $V_c$, we split the Brascamp-Lieb datum into two parts, $(\mathbf{L}_{V_c}, \mathbf{c})$ and $(\mathbf{L}_{V_c^\perp}, \mathbf{c})$ as follows:
\begin{align*}
  \mathbf{L}_{V_c} &= (V_c, \{L_j V_c\}_{1\leq j\leq m}, \{L_{j, {V_c}}\}_{1\leq j\leq m}),\\
  \mathbf{L}_{V_c^\perp}&=(H/V_c, \{H_j/(L_j V_c)\}_{1\leq j\leq m}, \{L_{j, {H/V_c}}\}_{1\leq j\leq m}),
\end{align*}
where $H/V_c = V_c^\perp$ and
\begin{align*}
	L_{j, V_c} &: V_c \to L_j V_c,\\
	L_{j, H/V_c} &: H/V_c \to H_j/(L_j V_c).
\end{align*}
In this paper, we choose $H = \R^d\times\R^k$ and $H_j = \R^d$.

\begin{lemma}[Lemma 4.6, \cite{BCCT2008}]\label{lem-split}
  Let $V_c$ be a critical subspace.
  Then $\mathrm{BL}(\mathbf{L}, \mathbf{c})$ is finite 
  if and only if $(\mathbf{L}_{V_c}, \mathbf{c})$ and $(\mathbf{L}_{V_c^\perp}, \mathbf{c})$ 
  satisfy \eqref{condi-transv} and \eqref{condi-scaling} for any subspace $V$ of $V_c$ and $V_c^\perp$, respectively.
\end{lemma}

Now we will prove Theorem ~\ref{thm-lp improving}. We first decompose $\mathcal{S}^k$ into a finite cover $\{\mathcal{S}_\tau^k\}$ for which $\mathcal{A}_{\mathcal{S}^k}(\rF)(x)$ can be written as a finite summation of the following operators.
\begin{align}
  \mathcal{A}_{\mathcal{S}_\tau^k}^\Theta(\rF)(x) = \int_{\R^k} \prod_{j=1}^m f_j (x + \Theta_j (y', \Phi_\tau(y')) \chi_\tau(y')~\mathrm{d}y',
\end{align}
where $\Phi_\tau : \R^k \to \R^{d-k}$ is a $C^2$-submersion and $\chi_\tau$ is a smooth cut-off function.

To simplify our proof, we consider a more general $m$-linear operator $T_K^{\vec{B}}$.
Suppose $B_j : \R^{d}\times \R^{k} \to \R^{d}$ be $C^2$ submersions
and $L_j = ~\mathrm{d}B_j(0, 0)$ with $j=1, \dots, m$.
Then $T_{K}^{\vec{B}}(\rF)$ is given by 
\begin{align}
  T_K^{\vec{B}}(\rF)(x) = \int_{\R^{k}} \prod_{j=1}^m f_j(B_j(x,y)) \, K(y) ~\mathrm{d}y, \quad y\in \R^{k}, x\in \R^{d},
\end{align}
where $K$ is a nonnegative bounded function supported in a ball $B(0,\varepsilon) \subset \R^{k}$.
Note that $\mathcal{A}_{\mathcal{S}_\tau^k}^\Theta(\rF)$ is an example of $T_K^{\vec{B}}$ for $K = \chi_\tau$ and $B_j(x,y') = x+\Theta_j(y', \Phi_\tau(y'))$.
Moreover, we take $c_j =1/p_j$ for $j=1, \dots, m$ and $c_{m+1} = 1/p'$ where $1/p = 1/p_1+\cdots+1/p_m-k/d$. 
Then we prove the following proposition.
\begin{proposition}\label{prop-improve esti}
  Let $\frac{1}{p} = \sum_{j=1}^m \frac{1}{p_j} - \frac{k}{d}$, $1 \le p \le \frac{d}{d-k}$, $p_j \in [1,\infty)$, $j=1, \dots, m$ satisfy $\sum_{j=1}^m\frac{1}{p_j} \geq1 $.
  Suppose $(\mathbf{L}, \mathbf{p})$ is a Brascamp-Lieb datum for 
  \begin{align*}
  	\mathbf{L} &= (\R^{d}\times\R^{k}, \{\R^{d}\}_{j=1}^{m+1},\{L_j\}_{j=1}^{m+1}),\\
	L_j &= \mathrm{d}B_j(0,0), j=1,\dots, m,\quad L_{m+1} = ~\mathrm{d}\pi_{\R^{d}},\\
  	\mathbf{p} &= \Big(\frac{1}{p_1}, \dots, \frac{1}{p_m}, \frac{1}{p'}\Big).
  \end{align*}
  Then we have 
  $$
  \| T_K^{\vec{B}}(\rF) \|_{L^p(\R^{d})} \lesssim (1+\varepsilon) \mathrm{BL}(\mathbf{L}, \mathbf{p}) \prod_{j=1}^m \|f_j\|_{L^{p_j}(\R^{d})}.
  $$
\end{proposition}
Proposition~\ref{prop-improve esti} states that the Brascamp-Lieb inequality implies an $L^p$ improving estimates.

\begin{proof}
Since $p\geq1$, by making use of duality we get 
\begin{align*}
	\| T_K^{\vec{B}}(\rF)\|_{L^p(\mathbb{R}^{d})} =  \sup_{\|g\|_{p'} \leq 1} \int_{\mathbb{R}^{d}} T_K^{\vec{B}}(\rF)(x) g(x) ~\mathrm{d}x.
\end{align*}
 Now we choose $g \in L^{p'}(\mathbb{R}^{d})$ such that $\|g\|_{L^{p'}(\mathbb{R}^{d})} \leq 1$ 
 As in Section~\ref{section 2}, we decompose $\R^d$ into countable union of cubes ${Q}_{\bf{n}}(\varepsilon)$, where $Q(\varepsilon)$ is a cube centered at the origin with side length $\varepsilon$ and ${Q}_{\bf{n}}(\varepsilon)$ denotes $\varepsilon{\bf{n}}$ translation of $Q(\varepsilon)$ for ${\bf{n}}\in\mathbb{Z}^d$.
 Then it follows that
\begin{align*}
	\int_{\mathbb{R}^{d}} T_K^{\vec{B}}(\rF)(x) g(x) ~\mathrm{d}x
	&=  \sum_{{\bf{n}}\in\mathbb{Z}^{d}} \int_{{Q}_{\bf{n}}(\varepsilon)} \int_{\R^{k}}  \Big(\prod_{j=1}^m f_j (B_j(x,y)) \Big)   g(x) K(y)  ~\mathrm{d}y~\mathrm{d}x \\
	&=  \sum_{{\bf{n}}\in\mathbb{Z}^{d}} \int_{[-\varepsilon/2,\varepsilon/2)^d} \int_{\R^{k}}  \Big(\prod_{j=1}^m f_j (B_j(x+\varepsilon{\bf{n}},y)) \Big)   g(x+\varepsilon{\bf{n}}) K(y)  ~\mathrm{d}y~\mathrm{d}x \\
	&= \sum_{{\bf{n}}\in\mathbb{Z}^{d}} \int_{[-\varepsilon/2,\varepsilon/2)^d} \int_{\R^{k}}  \Big(\prod_{j=1}^m \tau_{\varepsilon\bf{n}}[f_j] (B_j(x,y)) \Big)   \tau_{\varepsilon\bf{n}}[g](x) K(y)  ~\mathrm{d}y~\mathrm{d}x,
\end{align*}
where $\tau_{\varepsilon\bf{n}}[f](x) = f(x+\varepsilon{\bf{n}})$.
Then we apply Theorem \ref{thm-NLBL} to $\tau_{\varepsilon\bf{n}}[f_j]^{p_j}, \tau_{\varepsilon\bf{n}}[g]^{p'}$ together with additional mapping $L_{m+1} = ~\mathrm{d}\pi_{\R^d}$, which yields
\begin{align}\label{240102_1958}
	\int_{\mathbb{R}^{d}} T_K^{\vec{B}}(\rF)(x) g(x) ~\mathrm{d}x
	\leq(1+\varepsilon)  \mathrm{BL}(\mathbf{L}, \mathbf{p})  \sum_{{\bf{n}}\in\mathbb{Z}^{d}} \Big(\prod_{j=1}^m \|\tau_{\varepsilon\bf{n}}[f_j]\|_{L^{p_j}(\widetilde{Q}(\varepsilon))} \Big) \|\tau_{\varepsilon\bf{n}}[g]\|_{L^{p'}(\widetilde{Q}(\varepsilon))}
\end{align}
whenever
\begin{align}\label{condi-transv-pf}
	\dim(V) \leq \sum_{j=1}^m \frac{\dim(dB_j(0_x,0_y)(V))}{p_j} + \frac{\dim(~\mathrm{d}\pi_{\mathbb{R}^{d}}(V))}{p'}
\end{align}
for every subspace $V$ of $\R^{d}\times \R^{k}$ together with
\begin{align}\label{condi-scaling-pf}
	\frac{k}{d} + \frac{1}{p} = \sum_{j=1}^m \frac{1}{p_j}.
\end{align}
Note that $\varepsilon$ in \eqref{240102_1958} is uniform in $\bf{n}$ due to $B_j(x+\varepsilon{\bf{n}},y) = B_j(x,y)+\varepsilon{\bf{n}}$ and also that $\tilde{Q}$ denotes a cube whose side length is 3 times of that of $Q$ with the same center.

We choose $g$ such that $\|g\|_{p'}\leq1$, so we ignore $\|g\|_{L^{p'}(\widetilde{Q}_{\bf{n}})}$,
and that $\|\tau_{\varepsilon\bf{n}}[f_j]\|_{L^{p_j}(\widetilde{Q}(\varepsilon))} = \|f_j\|_{L^{p_j}(\widetilde{Q}_{\bf{n}}(\varepsilon))}$.
Thus by H\"older's inequality we have
\begin{align*}
	\sum_{{\bf{n}}\in\mathbb{Z}^{d}} \prod_{j=1}^m \|f_j\|_{L^{p_j}(\widetilde{Q}_{\bf{n}}(\varepsilon))} 
  \lesssim  \prod_{j=1}^m  \Big\| \|f_j\|_{L^{p_j}(\widetilde{Q}_{\bf{n}}(\varepsilon))} \Big\|_{\ell^{r_j}(\mathbb{Z}^d)}
\end{align*}
for $\sum_{j=1}^m \frac{1}{r_j} =1$ with $1\leq r_j \leq \infty$.
From $\sum_{j=1}^m \frac{1}{p_j} \geq1$, one can choose $r_j$'s such that $\frac{1}{r_j}\leq \frac{1}{p_j}$ for each $j=1, \dots, m$, 
then we use $\ell^{p_j} \hookrightarrow \ell^{r_j}$ embedding to obtain
\begin{align}\label{240102_2026}
\begin{split}
	\int_{\mathbb{R}^{d}} T_K^{\vec{B}}(\rF)(x) g(x) ~\mathrm{d}x
	&\leq (1+\varepsilon)  \mathrm{BL}(\mathbf{L}, \mathbf{p}) \prod_{j=1}^m \Big\| \|f_j\|_{L^{p_j}(\widetilde{Q}_{\bf{n}}(\varepsilon))} \Big\|_{\ell^{r_j}(\mathbb{Z}^d)} \\
	&\leq (1+\varepsilon)  \mathrm{BL}(\mathbf{L}, \mathbf{p})  \prod_{j=1}^m \Big\| \|f_j\|_{L^{p_j}(\widetilde{Q}_{\bf{n}}(\varepsilon))} \Big\|_{\ell^{p_j}(\mathbb{Z}^d)}.
\end{split}
\end{align}
Since $\widetilde{Q}_{\bf{n}}(\varepsilon)$ are finitely overlapped, taking supremum over $\|g\|_{p'}\leq1$ in \eqref{240102_2026} gives
\begin{align}
	\|T_K^{\vec{B}}(\rF)\|_{L^p(\mathbb{R}^{d})}\lesssim  (1+\varepsilon) \mathrm{BL}(\mathbf{L}, \mathbf{p})  \prod_{j=1}^m \| f_j\|_{L^{p_j}(\mathbb{R}^{d})}
\end{align}
for desired $p, p_1, \dots, p_m$.
\end{proof}

Now we present the proof of Theorem~\ref{thm-lp improving}.
\subsection{Proof of Theorem~\ref{thm-lp improving}}

There is a $C^2$ mapping $\Phi : \R^{k} \to \R^{d_c}$ for $d_c = d - k$ such that 
$\cS^k$ is locally a graph $\{(y', \Phi(y'))\in \R^{d}\}$.
Then in Proposition~\ref{prop-improve esti} we let $B_j(x,y') = x + \Theta_j (y', \Phi(y'))$ for $\Phi = (\phi^1, \dots, \phi^{d_c})$, $j=1, \dots, m$.
For $j = m+1$, let $B_{m+1} = \pi_{\R^d}$, where $\pi_{\R^d} : \R^d\times \R^k \to \R^{d}$ is a projection onto $x$-variable in $\R^{d}$.
For $j=1, \dots, m$, we define $L_j := ~\mathrm{d}B_j(0,0)$, which is given by
\begin{align}
  \begin{bmatrix}
    I_{d} &\Theta_j \nabla|_{y'=0}
    \begin{bmatrix}
      y'_1\\
      \vdots\\
      y'_{k}\\
      \phi^1(y')\\
      \vdots\\
      \phi^{d-k}(y')
    \end{bmatrix}
  \end{bmatrix}
  =
  \begin{bmatrix}
    I_{d} & \Theta_j
  \begin{bmatrix}
    1  &   0  &   \dots &   0  \\
    0  &   1  &   \dots &   0  \\
    \vdots&\vdots&\ddots&\vdots\\
    0  &   0  &   \dots &   1  \\
    \phi_{y'_1}^1(0) & \phi_{y'_2}^1(0)&\dots & \phi_{y'_{k}}^1(0)\\
    \vdots &\vdots &\ddots&\vdots\\
    \phi_{y'_1}^{d_c}(0) & \phi_{y'_2}^{d_c}(0) & \dots &\phi_{y'_{k}}^{d_c}(0)
  \end{bmatrix}
\end{bmatrix},
\end{align} 
where $I_{d}$ denotes the $d \times d$ identity matrix.
If there is no confusion, we simply write
\begin{align}
  L_j = 
  \begin{bmatrix}
    I_{d} & \Theta_j 
    \begin{bmatrix}
      I_{k}\\
      ~\mathrm{d} \Phi (0)
    \end{bmatrix}
  \end{bmatrix}
  ,\quad ~\mathrm{d}\Phi = (~\mathrm{d}\phi^1, \dots, ~\mathrm{d}\phi^{d_c}).
\end{align}
Without loss of generality we assume that $~\mathrm{d}\Phi(0)$ is a $d_c \times k$ zero matrix.
That is, we have for $j=1, \dots, m$
\begin{align}
  L_j = 
  \begin{bmatrix}
    I_{d} & \Theta_j^1
  \end{bmatrix}
  ,\quad \Theta_j = 
  \begin{bmatrix}
    \Theta_j^1 & \Theta_j^2
  \end{bmatrix}
  ,
\end{align}
where $\Theta_j^1$ and $\Theta_j^2$ are $d\times k$ and $d\times d_c$ matrices, respectively.
Since $\Theta_j^1$ has $k$ linearly independent columns, its rank is $k$.
In the case of $j=m+1$, we have
\begin{align}
~\mathrm{d}\pi_{\R^d} = 
\begin{bmatrix}
  I_{d} & Z_{d}
\end{bmatrix},
\end{align}
where $Z_{d}$ means all $d \times d$ elements are zero.
We show that $\mathbf{L} = (L_1, \dots, L_m, ~\mathrm{d}\pi_{\R^d})$ and $\mathbf{p} = (\frac{1}{m}, \dots, \frac{1}{m}, \frac{k}{d})$ 
are Brascamp-Lieb data by making use of Lemma~\ref{lem-split}.

Let $K_\pi = \ker(\pi_{\R^d})= \{(0, y') \in \R^d \times \R^k : y' \in \R^k \}$, which is $k$-dimensional.
Then it is clear that $K_\pi^\perp = \{(x, 0) \in \R^d \times \R^k : x \in \R^d \}$.
For $K_\pi$ we have
\begin{align*}
  \sum_{j=1}^m \frac{\dim(L_j K_\pi) }{p_j} + \frac{\dim(\pi_{\R^{d}}K_\pi)}{p'}
  = \sum_{j=1}^m \frac{k}{p_j} 
  = k 
  =\dim(K_\pi).
\end{align*}
Since $\dim(L_j K)$ equals to $\dim(K)$ for any subspace $K$ of $K_\pi$ with $j=1, \dots, m$, 
we also have 
\begin{align*}
  \sum_{j=1}^m \frac{\dim(L_j K) }{p_j} + \frac{\dim(\pi_{\R^{d}}K)}{p'}
  = \sum_{j=1}^m \frac{\dim(K)}{p_j} 
  =\dim(K).
\end{align*}
Thus $K_\pi$ is a critical subspace and $(\mathbf{L}_{K_\pi}, \mathbf{p})$ is a Brascamp-Lieb datum. 

On the other hand, for $K_\pi^\perp = \{(x, 0) \in \R^d \times \R^k : x\in \R^d\}$ we consider $(\mathbf{L}_{K_\pi^\perp}, \mathbf{p})$
\begin{align*}
  \mathbf{L}_{K_\pi^\perp} &= (K_\pi^\perp, \{\R^d / (L_j K_\pi)\}_{1\leq j\leq m+1}, \{L_{j, {K_\pi^\perp}}\}_{1\leq j \leq m+1}),\\
  \mathbf{p} & = \Big(\frac{1}{p_1}, \dots, \frac{1}{p_m}, \frac{1}{p'}\Big).
\end{align*}
Note that $\pi_{\R^d, K_\pi^\perp} = \pi_{\R^d}$.
Then we have
\begin{align*}
  &\sum_{j=1}^m \frac{\dim(L_{j,K_\pi^\perp} K_\pi^\perp) }{p_j} + \frac{\dim(\pi_{\R^{d}}K_\pi^\perp)}{p'}\\
  &= \sum_{j=1}^m \frac{d-k}{p_j} + \frac{d}{p'}\\
  &= d-k + \frac{d k }{d} =d
  =\dim(K_\pi^\perp).
\end{align*}
It remains to verify \eqref{condi-transv} for any proper subspace of $K_\pi^\perp$. In order to show this, we consider a subspace $K$ of $K_\pi^\perp$ whose dimension $d_K$ satisfies $d> d_K \geq k$ or $k> d_K\geq1$.
Note that it is important to check dimension of $L_j K / L_j K_\pi$, but it suffices to consider $K$ instead of $L_j K$ because every element of $K$ is given by $(x,0)\in \R^d\times \R^k$ and $L_j (x,0) = x$.

\subsubsection{$d>d_K > k$}
Let $K$ be a $d_K$ dimensional subspace of $K_\pi^\perp$ and observe that
\begin{align*}
	L_\mu (0,y') &= \Theta_\mu^1 y', \quad\forall (0, y') \in K_{\pi}.
\end{align*}
Then we define $K_j := L_j K_\pi = \{(\Theta_j^1 y',0)  \in \R^d \times \R^k : y' \in \R^k \}$ is a $k$ dimensional subspace of $K_\pi^\perp$ due to full rank of $\Theta_j^1$.
It is possible that for some $\mu=1, \dots, m$, $L_\mu K \cap K_\mu = K\cap K_\mu$ is $k$ dimensional.
Thus in general, we have $\dim(L_{\mu,K_\pi^\perp} K) = \dim(K/K_\mu) \geq d_K - k$.
Note that our choice of $\{\Theta_j\}$ satisfying \eqref{ineq-240115 1454} allows us to have
\begin{align*}
	\dim\Big(\text{span}\big\{ K_\mu,  K_{j_1}, \dots, K_{j_\ell}\big\}\Big) \geq k +\ell,\quad j_i \not= \mu.
\end{align*}
That is, there are at most $\ell = d_K -k$ $j$'s such that $K_j$ is a subspace of $K$.
Therefore we have $\dim(K/K_j) \geq d_K - k$ for those at most $\ell +1= d_K -k+1$ $j$'s.
Otherwise, for the rest $m-\ell-1$ $j$'s we have $\dim(K/K_j)\geq d_K -k +1$.
Hence it follows that
\begin{align}\label{240118_1606}
\begin{split}
  &\sum_{j=1}^m \frac{\dim(L_{j,K_\pi^\perp} K) }{p_j} + \frac{\dim(\pi_{\R^{d}}K)}{p'}\\
  \geq &\frac{(d_K -k)(\ell+1)}{m} + \frac{(d_K - k+1)(m-\ell-1)}{m}+ \frac{k}{d}d_K\\
  = &(d_K -k)+ \frac{m-\ell-1}{m}+ \frac{k}{d}d_K\\
  = &d_K -k+ \frac{m-d_K +k-1}{m}+ \frac{k}{d}d_K.
 \end{split}
\end{align}
Here we choose $p_j = m$ for all $j=1, \dots, m$ in order to minimize the loss, $i.e.$ to maximize the lower bound of \eqref{240118_1606}.
Thus we fix $\mathbf{p}$ as
$$
 \Big(\frac{1}{p_1}, \dots, \frac{1}{p_m}, \frac{1}{p'}\Big) = \Big(\frac{1}{m}, \dots, \frac{1}{m}, \frac{k}{d}\Big).
$$
Note that the last line of \eqref{240118_1606} is greater than or equal to $d_K$ whenever 
\begin{align}\label{ineq-230106 1345}
 -k+ \frac{m-d_K +k-1}{m}+ \frac{k}{d}d_K \geq 0.
\end{align}
From $d>d_K>k$, it follows that the left-hand side of \eqref{ineq-230106 1345} is larger than
\begin{align*}
	-k + \frac{m-d+1+k-1}{m} + \frac{k(k+1)}{d}.
\end{align*}
Thus \eqref{ineq-230106 1345} holds whenever 
\begin{align*}
	\frac{m-d+k}{m} \geq \frac{d-k-1}{d}k.
\end{align*}

\subsubsection{$k \geq d_K\geq1$}

\begin{itemize}

\item $d_K = k$.

Let $K$ be not equal to any $K_j$ for $j=1, \dots, m$ so that $\dim(K \cap K_j) \leq k-1$.
Thus we have $\dim(L_{j, K_\pi^\perp}K) = \dim(K/K_j)\geq 1$ for $j=1, \dots, m$, so it follows that
\begin{align*}
  \sum_{j=1}^m \frac{\dim(L_{j,K_\pi^\perp} K) }{p_j} + \frac{\dim(\pi_{\R^{d}}K)}{p'}
  \geq &\sum_{j=1}^m\frac{1}{m} + \frac{k^2}{d}\\
  = &1+\frac{k^2}{d}\\
  = &k +1-  \frac{d-k}{d}k.
\end{align*}
The last line is greater than or equal to $\dim(K)=k$ if
\begin{align}\label{k-0_1}
	1 \geq  \frac{d-k}{d}k.
\end{align}

On the other hand, let $K=K_\mu$ for some $\mu=1, \dots, m$. 
By \eqref{ineq-230106 1631}, we have $\dim(K\cap K_j)\leq k-1$ for $j\not=\mu$.
Thus it follows that
\begin{align*}
  \sum_{j=1}^m \frac{\dim(L_{j,K_\pi^\perp} K) }{p_j} + \frac{\dim(\pi_{\R^{d}}K)}{p'}
  &\geq \sum_{j\not=\mu} \frac{1}{m} + \frac{k^2}{d}\\
  &=\frac{m-1}{m} + \frac{k^2}{d}\\
  &= k +  \frac{m-1}{m} -  \frac{d-k}{d}k.
\end{align*}
The last line is greater than or equal to $\dim(K)=k$ if
\begin{align}\label{k-0_2}
	\frac{m-1}{m} \geq  \frac{d-k}{d}k.
\end{align}
Note that \eqref{k-0_1} is implied by \eqref{k-0_2}.

\item $d_K=k-1$.

For $d_K = k-1$, we consider a subspace $K$ such that $K$ is not contained in $K_j$ for all $j=1,\dots, m$.
Then for some $\mu\in\{1,\dots, m\}$ the worst case in verifying \eqref{condi-transv} is that $K \cap K_\mu$ is $k-2$ dimensional, 
since $\dim(L_{j,K_\pi^\perp} K)$ gets lower as $\dim(K \cap K_\mu)$ gets larger.
Thus we have $\dim(K/K_\mu)=1$, and this may happen for any $j=1, \dots, m$.
It follows that
\begin{align*}
  \sum_{j=1}^m \frac{\dim(L_{j,K_\pi^\perp} K) }{p_j} + \frac{\dim(\pi_{\R^{d}}K)}{p'}
  \geq &\sum_{j=1}^m\frac{1}{m} + \frac{k}{d}(k-1)\\
  = &1+\frac{k}{d}(k-1)\\
  = &(k-1) +1-  \frac{d-k}{d}(k-1).
\end{align*}
The last line is greater than or equal to $\dim(K)=k-1$ if
\begin{align}\label{k-1_1}
	1 \geq  \frac{d-k}{d}(k-1).
\end{align}

Now, let $K$ is a $k-1$ dimensional subspace of $K_\mu$ for some $\mu\in\{1, \dots, m\}$.
Then the worst case is that $K$ is given by intersection of $K_\mu$ and $K_\nu$ for some $\nu\not=\mu$.
Thus we have $\dim(K/K_\mu) = \dim(K/K_\nu)=0$.
However, if we choose any other $j\not= \mu, \nu$, then we have $\dim(K/K_j)\geq1$ because $\dim(K_\mu \cap K_\nu \cap K_j)\leq k-2$ due to \eqref{ineq-230106 1631}.
Without loss of generality, say $\mu=1$ and $\nu=2$ so that by \eqref{ineq-230106 1631} one can check
\begin{align*}
  \sum_{j=1}^m \frac{\dim(L_{j,K_\pi^\perp} K) }{p_j} + \frac{\dim(\pi_{\R^{d}}K)}{p'}
  &\geq \sum_{j=3}^{m} \frac{1}{m} + \frac{k}{d}(k-1)\\
  &=\frac{m-2}{m} + \frac{k}{d}(k-1)\\
  &= (k-1) +  \frac{m-2}{m} -  \frac{d-k}{d}(k-1).
\end{align*}
The last line is greater than or equal to $k-1$ whenever
\begin{align}\label{k-1_2}
	\frac{m-2}{m} \geq \frac{d-k}{d}(k-1).
\end{align}
Note that \eqref{k-1_2} implies \eqref{k-1_1}.

\item $d_K = k-n$.

Similar to $k, k-1$ dimensional cases of $K$, for an arbitrary $k-n$ dimensional subspace $K$, one can check that the worst case happens when $K$ is contained in $K_{j_i}$ for $j_1, \dots, j_n$.
Thus we have
\begin{align*}
  \sum_{j=1}^m \frac{\dim(L_{j,K_\pi^\perp} K) }{p_j} + \frac{\dim(\pi_{\R^{d}}K)}{p'}
  &\geq \sum_{j\not=\mu, j_1, \dots, j_n} \frac{1}{m} + \frac{k}{d}(k-n)\\
  &= \frac{(m-n-1)}{m} +  \frac{k}{d}(k-n)\\
  &= (k-n) + \frac{(m-n-1)}{m} -  \frac{d-k}{d}(k-n).
\end{align*}
Then the last line is greater than or equal to $k-n$ whenever
\begin{align}\label{ineq-230106 1551}
  \frac{(m-n-1)}{m} \geq  \frac{d-k}{d}(k-n),
\end{align}
which leads to $m\geq d$ when $k=d-1$ or $k=n+1$.
\end{itemize}

Together with \eqref{ineq-230106 1345}, one can conclude that $\mathrm{BL}(\mathbf{L}, \mathbf{p})$ is finite for given data $(\mathbf{L}, \mathbf{p})$ whenever $\frac{m-n-1}{m}\geq \frac{d-k}{d}(k-n)$ for all $0\leq n \leq k-1$.
Note that we can rewrite the condition as
\begin{align}\label{ineq-230106 2006}
  \frac{m-1}{m} \geq \frac{d-k}{d}k - \Big(\frac{d-k}{d} - \frac{1}{m}\Big)n,\quad 0\leq n \leq k-1.
\end{align}
Note that \eqref{ineq-230106 2006} is reduced to 
\[
	\frac{m-1}{m} \geq \frac{d-k}{d}k,
\]
for all $0\leq n\leq k-1$ when $m\geq d$. 
Thus, $(\mathbf{L}, \mathbf{p})$ is a Brascamp-Lieb datum. Hence, by Proposition \ref{prop-improve esti} we prove Theorem~\ref{thm-lp improving}.

\section{Proof of Theroem~\ref{thm-lac}}\label{sec_thm_lac}

Recall that the lacunary maximal function $\mathcal{M}^{\Theta}_{\mathcal{S}}$ is defined by 
\begin{align}
  \mathcal{M}^{\Theta}_{\mathcal{S}}(\rF)(x)
  =\sup_{\ell\in\mathbb{Z}} \Big|\int_{\mathcal{S}} \prod_{j=1}^m f_{j}(x-2^{-\ell}\Theta_j y) ~~\mathrm{d}\sigma(y)\Big|,
\end{align}
where $\mathcal{S}$ has $\kappa$-nonvanishing principal curvatures and $\Theta = \{\Theta_j\}$ is a family of mutually linearly independent rotation matrices.

Observe that for any fixed $\ell\in\mathbb{Z}$, we can write the identity operator $I$ as follows 
 \begin{align}
  I=P_{<\ell}+\sum^{\infty}_{n=0}P_{\ell+n} = P_{<\ell} + P_{\ell\leq }.\label{proj_identity}
 \end{align}
 Then we have
 \begin{equation}\label{proj_mlinear_id}
 \begin{aligned}
  \prod_{j=1}^m f_j 
  &= \prod_{j=1}^m \Big(P_{<\ell}f_j + P_{\ell\leq}f_j \Big)\\
  &= \Bigl(\prod_{j=1}^m P_{<\ell}f_j \Bigr) 
  +\Bigl( \prod_{j=1}^m P_{\ell\leq}f_j\Bigr)\\
  &\quad + \sum_{\alpha= 1}^{m-1} \frac{1}{\alpha! (m-\alpha)!} \sum_{\tau \in S_m}\Bigl(\prod_{i=1}^\alpha P_{\ell\leq}f_{\tau(i)} \Bigr) \Bigl( \prod_{i=\alpha+1}^m P_{<\ell}f_{\tau(i)}\Bigr),
   \end{aligned}
\end{equation}
where the second summation runs over the symmetric group $S_m$ over $\{1, \dots, m\}$.
For ${\bf{n}}=(n_1,\cdots,n_m) \in \mathbb{N}_0^m = (\mathbb{N} \cup \{0\})^m$, we define 
\begin{align}
	\mathcal{A}_{\ell}^{\alpha,\tau}(\rF)(x)
    &:=\int_{\mathcal{S}} \Bigl(\prod_{i=1}^\alpha P_{<\ell}f_{\tau(i)}(x-2^{-\ell}\Theta_{\tau(i)} y) \Bigr) \Bigl( \prod_{i=\alpha+1}^m f_{\tau(i)} (x - 2^{-\ell}\Theta_{\tau(i)} y)\Bigr)~~\mathrm{d}\sigma(y),\label{A_ell_F}\\
   \tilde{\mathcal{A}}_{\ell}^{\alpha,\tau}(\rF)(x)
    &:=\int_{\mathcal{S}} \Bigl(\prod_{i=1}^\alpha f_{\tau(i)}(x-2^{-\ell}\Theta_{\tau(i)} y) \Bigr) \Bigl( \prod_{i=\alpha+1}^m P_{<\ell} f_{\tau(i)} (x - 2^{-\ell}\Theta_{\tau(i)} y)\Bigr)~~\mathrm{d}\sigma(y)\\
	\mathcal{M}_{\bf{n}}(\rF)
    &:=\sup_{\ell\in\mathbb{Z}} \Big|\int_{\mathcal{S}} \prod_{j=1}^m P_{\ell+n_j} f_j(x - 2^{-\ell}\Theta_j y)~~\mathrm{d}
    \sigma(y) \Big|,\label{M_n_F}\\
	\mathcal{S}_{\bf{n}}(\rF)
    &:=\sum_{\ell\in\mathbb{Z}}\Big|\int_{\mathcal{S}} \prod_{j=1}^m P_{\ell+n_j} f_j(x - 2^{-\ell}\Theta_j y)~~\mathrm{d}
    \sigma(y) \Big|.\label{S_n_F}
\end{align}
Note that $\mathcal{M}_{\bf{n}}(\rF)$ corresponds to $\alpha=0$ case in \eqref{proj_mlinear_id}.
Therefore, the lacunary maximal function $\mathcal{M}^{\Theta}_{\mathcal{S}}$ can be controlled by a constant mutiple of 
\begin{align}
  \sum_{\alpha=1}^m \sum_{\tau\in S_m}\sup_{\ell\in\Z} \Big(|\mathcal{A}_\ell^{\alpha,\tau}(\rF)| +|\tilde{\mathcal{A}}_\ell^{\alpha,\tau}(\rF)| \Big)+ \sum_{\mathbf{n}\in \mathbb{N}_0^m} \mathcal{M}_{\mathbf{n}}(\rF).
\end{align}
By similarity of $\mathcal{A}_\ell^{\alpha,\tau}(\rF)$ and $\tilde{\mathcal{A}}_\ell^{\alpha,\tau}(\rF)$ together with the symmetry on $\tau\in S_m$, instead of the first summation it suffices to consider estimates for $\mathcal{A}_\ell^\alpha(\rF)$, which is given by
\begin{align}
	\mathcal{A}_{\ell}^{\alpha}(\rF)(x)
    &:=\int_{\mathcal{S}} \Bigl(\prod_{j=1}^\alpha P_{<\ell}f_{j}(x-2^{-\ell}\Theta_j y) \Bigr) \Bigl( \prod_{j=\alpha+1}^m f_{j} (x - 2^{-\ell}\Theta_j y)\Bigr)~~\mathrm{d}\sigma(y)
\end{align}
Then the proof will be completed by combination of the following lemmas and induction on $m$-linearity:
\begin{lemma}\label{lem-m2}
  For $m=2$ and $\alpha=1$ we have
  $$
    \mathcal{A}_\ell^\alpha(\rF)(x) \leq M_{HL}(f_1)(x) \times M^{\Theta_2}_\mathcal{S}(f_2)(x),
  $$
  where $\rF = (f_1,f_2)$ and 
  \begin{align*}
  M^{\Theta_2}_\mathcal{S}(f_2)(x) = &\sup_{\ell\in\Z}|\int_{\mathcal{S}} f_2(x-2^{-\ell}\Theta_2y)~~\mathrm{d}\sigma(y)|\\
  =&\sup_{\ell\in\Z}|\int_{\mathcal{S}} f_2\Big(\Theta_2(\Theta^{-1}_2x-2^{-\ell}y)\Big)~~\mathrm{d}\sigma(y)|.
  \end{align*}
\end{lemma}
\begin{proof}
For $m=2$ we have 
\begin{align*}
	\mathcal{A}_\ell^\alpha(\rF)(x)  = \Big| \int_{\mathcal{S}} P_{<\ell}f_1(x-2^{-\ell}\Theta_1y) \times f_2(x-2^{-\ell}\Theta_2y) ~\mathrm{d}\sigma(y)\Big|
\end{align*}
  It suffices to show $\sup_{y\in\mathcal{S}}|P_{<\ell}f(x-2^{-\ell}y)| \lesssim M_{HL}(f)(x)$, where $M_{HL}$ denotes the Hardy-Littlewood maximal function. Since $\phi_\ell(x) = 2^{\ell d}\phi(2^\ell x)$, we have
  \begin{align*}
    P_{<\ell}f(x-2^{-\ell}y)
    &= \int_{\R^d} f(z) 2^{\ell d}\phi(2^\ell(x - 2^{-\ell}y -z))~~\mathrm{d}z\\
    &= \int_{\R^d} f(x+ 2^{-\ell}z) \phi(y-z)~~\mathrm{d}z.
  \end{align*}
  Since $y$ is contained in a compact surface $\mathcal{S}$, we have for any $N>0$
  $$
    |P_{<\ell}f(x-2^{-\ell}y)| \lesssim \int_{\R^d} |f(x + 2^{-\ell}z)| \frac{C_N}{(1+|z|)^N}~~\mathrm{d}z \leq M_{HL}(f)(x).
  $$
\end{proof}
Since $M_{HL}, M^{\Theta_2}_\mathcal{S}$ are bounded on $L^p$ for $p\in(1, \infty]$, we need to handle the summation of $M_\mathbf{n}$ over $\mathbf{n} \in \mathbb{N}_0^m$.
Note that for $\alpha=2$ we have $\mathcal{A}_\ell^\alpha(\rF)(x) \lesssim M_{HL}(f_1)(x) \times M_{HL}(f_2)(x)$.

\begin{lemma}\label{lem-qba esti}
  Let $\mathbf{n} \in \mathbb{N}^m$ and $\frac{1}{p} = \frac{2(\kappa+1)}{\kappa+2}$.
  Then for $(\frac1{p_1},\cdots,\frac1{p_m}) \in \text{conv}(\mathcal{V}_{\kappa})$,  we have
  $$
    \| \mathcal{M}_{\mathbf{n}}(\rF) \|_{L^{p, \infty}} \leq C(1+|\mathbf{n}|^{m}) \prod_{j=1}^m \|f_j\|_{L^{p_j}}.
  $$
  In particular, we have $\frac{1}{p}=\frac{2d}{d+1}$ when we consider averages over $\mathcal{S} = \mathbb{S}^{d-1}$.
\end{lemma}

\begin{lemma}\label{lem-smoothing}
  Let $\mathbf{n} \in \mathbb{N}^m$ and $1= \sum_{j=1}^m \frac{1}{r_j}$ for some $r_1, \dots, r_m \in (1, \infty)$. Then we have
  $$
    \|\mathcal{S}_{\mathbf{n}}(\rF)\|_{L^1} \lesssim 2^{-\delta|\mathbf{n}|} \prod_{j=1}^m \|f_j\|_{L^{r_j}}.
  $$
\end{lemma}
Proofs of Lemmas~\ref{lem-qba esti} and \ref{lem-smoothing} will be given in Section~\ref{sec-qba n smthng} and note that Lemma~\ref{lem-smoothing} is an easy consequence of the assumption \eqref{ineq_reg_cM'}.
Since $\mathcal{M}_\mathbf{n} \leq \mathcal{S}_\mathbf{n}$ by definition, it follows from interpolation between Lemmas~\ref{lem-qba esti} and \ref{lem-smoothing} that
\begin{align}
  \| \mathcal{M}_\mathbf{n}(\rF) \|_{L^p} \lesssim 2^{-\delta'|\mathbf{n}|} \prod_{j=1}^m\|f_j\|_{L^{p_j}},\quad \delta'>0,
\end{align}
whenever $\frac{1}{p} <\frac{2(\kappa+1)}{\kappa+2}$ and $(1/p_1, \dots, 1/p_m)$ is in an interior of convexhull of $\textrm{conv}(\mathcal{V}_\kappa)$ and $(1/r_1, \dots, 1/r_m)$.
Since $2^{-\delta'|\mathbf{n}|}$ is summable over $\mathbf{n}\in\mathbb{N}_0^{m}$, this proves the theorem for $m=2$ case inside of the convexhull.
Then together with interpolation with trivial $L^\infty\times\cdots\times L^\infty \to L^\infty$ estimates, we prove the theorem for $m=2$.

For the induction, we assume that Theorem~\ref{thm-lac} holds for $N$-linear operators with $N=2,\cdots,m-1$. 
Note that we already showed that $m=2$-case holds.
By the assumption, we have the following lemma:
\begin{lemma}\label{lem-induction}
  For $\alpha=1, \dots, m$, we have
  $$
    \mathcal{A}_\ell^\alpha(\rF)(x) 
    \lesssim 
    \prod_{\mu=1}^\alpha M_{HL}(f_\mu)(x) \times 
    \sup_{\ell\in\Z}\int_{\mathcal{S}}\Big| \prod_{\nu=\alpha+1}^m f_\nu(x - 2^{-\ell}\Theta_\nu y)\Big|~~\mathrm{d}\sigma(y).
  $$
  Moreover, if we assume that Theorem~\ref{thm-lac} holds for $N$-linear operators with $N=2, 3, \dots, m-1$, then it follows that
  $$
  \sup_{\ell\in\Z}\int_{\mathcal{S}} \Big| \prod_{\nu=\alpha+1}^m f_\nu(x - 2^{-\ell}\Theta_\nu y)\Big|~~\mathrm{d}\sigma(y)
  $$
  satisfies multilinear estimates of Theorem~\ref{thm-lac}.
\end{lemma}
\begin{proof}
	The first assertion of the lemma follows directly by the proof of Lemma~\ref{lem-m2}.
	For the second assertion, it is just an $(m-\alpha)$-sublinear average, hence the conclusion follows directly by the assumption.
\end{proof}

We assume that Theorem~\ref{thm-lac} is true for $N$-linear operators with $N=2,\cdots,m-1$ and prove the $N=2$ case.
For general $m$, by Lemma~\ref{lem-induction} we have
\begin{align}\label{230102_1603}
\begin{split}
	&\|\mathcal{A}_\ell^\alpha(\rF)(x) \|_{L^{p}(\R^d)}\\
	\lesssim &\Big\|\prod_{\mu=1}^\alpha M_{HL}(f_\mu)(x) \times \sup_{\ell\in\Z}\int_{\Sigma} \Big|\prod_{\nu=\alpha+1}^m f_\nu(\cdot - 2^{-\ell}y_\nu)\Big|~\mathrm{d}\sigma(y) \Big\|_{L^{p}(\R^d)}\\
	\leq &\Big\|\prod_{\mu=1}^\alpha M_{HL}(f_\mu)\Big\|_{L^{1/\alpha_1}(\R^d)}
	\times \Big\|  \sup_{\ell\in\Z}\int_{\Sigma} \Big|\prod_{\nu=\alpha+1}^m f_\nu(\cdot - 2^{-\ell} y_\nu)\Big|~\mathrm{d}\sigma(y)\Big\|_{L^{1/\alpha_2}(\R^d)}\\
	\lesssim &\prod_{\mu=1}^\alpha \| f_\mu\|_{L^{p_\mu}(\R^d)}
	\times \prod_{\nu=\alpha+1}^m \| f_\nu \|_{L^{p_\nu}(\R^d)},
\end{split}
\end{align}
where $\alpha_1 = 1/p_1 +\cdots+1/p_\alpha$ and $\alpha_2 = 1/p_{\alpha+1} +\cdots+ 1/p_m$.

Since we already proved Lemmas~\ref{lem-qba esti} and \ref{lem-smoothing} for general $m$, together with \eqref{230102_1603} we show that Theorem~\ref{thm-lac} holds for $m$-linear lacunary maximal averages under the assumption that $N$ cases hold for $N=2,\dots,m-1$.
This closes the induction hence proves the theorem.

\section{Proofs of Lemmas~\ref{lem-qba esti} and \ref{lem-smoothing}}\label{sec-qba n smthng}

\subsection{Proof of Lemma~\ref{lem-qba esti}}

For the $L^{p,\infty}$-estimate of Lemma~\ref{lem-qba esti}, we assume $\|f_j\|_{p_j}=1$ and show the following inequality:
\begin{align}\label{ineq-levelset esti}
  \meas\Bigl(\{x : \mathcal{M}_{\mathbf{n}}(\rF)(x) > \lambda\}\Bigr) \lesssim |\mathbf{n}|^{m}\lambda^{-p}.
\end{align}
To obtain \eqref{ineq-levelset esti}, we exploit the approach of Chirst and Zhou \cite{Christ_Zhou2022}, which is based on the multilinear Calder\'on-Zygmund decomposition.
We apply the Calder\'on-Zygmund decomposition at height $C\lambda^{\frac{p}{p_j}}$ to each $f_j$, $j=1, \dots, m$ for some $C>0$ so that for each $j$ we have $f_j = g_j + b_j$ such that
\begin{align}
  &\|g_j\|_{\infty} \leq C \lambda^{\frac{p}{p_j}},\label{CZ-height}\\
  &b_j = \sum_{\gamma} b_{j, \gamma},\quad \supp(b_{j,\gamma})\subset Q_{j, \gamma},\label{CZ-bad}\\
  &\|b_{j, \gamma}\|_{L^{p_j}}^{p_j} \lesssim \lambda^{p}\meas(Q_{j, \gamma}), \quad \sum_{\gamma}\meas(Q_{j, \gamma}) \lesssim \lambda^{-p},\label{CZ-bad esti}\\
  &\int b_{j, \gamma} = 0\label{CZ-cancel}.
\end{align}
Note that $Q_{j,\gamma}$ denotes a dyadic cube.
Then we have
\begin{align*}
  \meas\Bigl(\{x : \mathcal{M}_{\mathbf{n}}(\rF)(x) > \lambda\}\Bigr)
  \lesssim &\meas\Bigl(\{x : \mathcal{M}_{\mathbf{n}}(g_1, \dots, g_m)(x) > 2^{-m}\lambda\}\Bigr)\\
  &+ \meas\Bigl(\{x : \mathcal{M}_{\mathbf{n}}(g_1, \dots, g_{m-1}, b_m)(x) > 2^{-m}\lambda\}\Bigr)\\
  &+\cdots+ \meas\Bigl(\{x : \mathcal{M}_{\mathbf{n}}(b_1, \dots, b_m)(x) > 2^{-m}\lambda\}\Bigr).
\end{align*}
For $C_\mathcal{S} = 5\max(1, \diam(\mathcal{S}))$ we define $\mathcal{E} = \cup_{j=1}^m \cup_{\gamma} C_\mathcal{S}Q_{j, \gamma}$ so that $\meas(\mathcal{E}) \lesssim \lambda^{-p}$.
Note that $C_\mathcal{S} Q$ is a cube whose side-length is $C_\mathcal{S}$ times of that of $Q$ with the same center as $Q$.
Thus we estimate each level set for $x\in \R^d\setminus\mathcal{E}$.

\subsubsection{Estimates for $\mathcal{M}_{\mathbf{n}}(b_1, \dots, b_m)$}\label{sssec-bad}

Let $b_j^i = \sum_{\gamma : s(Q_{j, \gamma}) = 2^{-i}} b_{j, \gamma}$ where $s(Q)$ denotes a side length of $Q$.
Then $| \mathcal{M}_{\mathbf{n}}(b_1, \dots, b_m) |^p$ with $p = \frac{\kappa+2}{2(\kappa+1)}$ is bounded by
\begin{align}
  \sum_{i_1, \dots, i_m \in \Z} \sum_{\ell\in\Z} | \mathcal{A}_{\ell}^\mathbf{n}(b_1^{i_1}, \dots, b_m^{i_m})|^p,
\end{align}
where 
$$
\mathcal{A}_{\ell}^\mathbf{n}(f_1, \dots, f_m)(x) = \int_{\mathcal{S}} \prod_{j=1}^m P_{n_j +\ell} f_j(x - 2^{-\ell}\Theta_j y)~~\mathrm{d}\sigma(y).
$$
To proceed further, we introduce the following lemmas whose proofs will be given in the last part of this subsection:
\begin{lemma}\label{lem-decay for b}
  For $(\frac{1}{p_1}, \dots, \frac{1}{p_m})\in \text{conv}(\mathcal{V}_{\kappa})$ with $\frac{1}{p} = \sum_{j=1}^m \frac{1}{p_j}$, we have
  $$
  \| \mathcal{A}_\ell^\mathbf{n}(b_{1}^{i_1}, \dots, b_{m}^{i_m}) \|_{L^p(\R^d\setminus\mathcal{E})}^p
  \lesssim \min_{j=1, \dots, m} \min(1, 2^{(n_j+\ell - i_j)(1+\frac{d}{p_j'})}, 2^{i_j - \ell}) \prod_{j=1}^m \|b_j^{i_j}\|_{p_j}^p.
  $$
\end{lemma}
For $m=2$ and $p=\frac{1}{2}$, Lemma~\ref{lem-decay for b} is given in \cite{Christ_Zhou2022}. 
The proof for general $m\geq2$ and $p = \frac{\kappa+2}{2(\kappa+1)}$ case is given in similar manner.

\begin{lemma}\label{lem-sum}
Under the same condition in Lemma~\ref{lem-decay for b},
  $$
    \sum_{\ell\in\Z} \min_{i_1, \dots, i_m} \min(1, 2^{(n_j+\ell - i_j)(1+\frac{d}{p_j'})}, 2^{i_j - \ell}) 
    \lesssim |\mathbf{n}| \prod_{j, j', j\not=j'} \min(1, 2^{|\mathbf{n}| - |i_j - i_{j'}|})^{\frac{1}{m(m-1)}}.
  $$
\end{lemma}
By using Lemmas~\ref{lem-decay for b} and \ref{lem-sum}, we have
\begin{align*}
  \| \mathcal{M}_\mathbf{n}(b_1, \dots, b_m) \|_{L^p(\R^d\setminus\mathcal{E})}^p
  &\lesssim \sum_{i_1 \in \Z} \cdots \sum_{i_m \in\Z} \sum_{\ell \in\Z} \| \mathcal{A}_\ell^\mathbf{n}(b_{1}^{i_1}, \dots, b_{m}^{i_m}) \|_{L^p(\R^d\setminus\mathcal{E})}^p\\
  &\lesssim \sum_{i_1, \dots, i_m}  |\mathbf{n}| \prod_{j, j', j\not=j'} \min(1, 2^{|\mathbf{n}| - |i_j - i_{j'}|})^{\frac{p}{m(m-1)}} \prod_{j=1}^m \|b_j^{i_j}\|_{p_j}^p.
\end{align*}
We apply H\"older's inequality to the last line and obtain
\begin{align*}
  \| \mathcal{M}_\mathbf{n}(b_1, \dots, b_m) \|_{L^p(\R^d\setminus\mathcal{E})}^p
  \lesssim |\mathbf{n}| \prod_{j=1}^m \Bigl( \sum_{i_1, \dots, i_m} \prod_{l\not=j} \min(1, 2^{|\mathbf{n}| - |i_j - i_{l}|})^{\frac{p_j}{m(m-1)}} ~ \|b_j^{i_j}\|_{p_j}^{p_j} \Bigr)^{\frac{p}{p_j}}.
\end{align*}
Oberserve that the summation over $i_1, \dots,i_{j-1}, i_{j+1},\dots, i_m$ yields
$$
  \sum_{i_j\in\Z} |\mathbf{n}|^{m-1} \|b_j^{i_j}\|_{p_j}^{p_j}.
$$
It is because we have $\sum_{i_l} \min(1, 2^{|\mathbf{n}| - |i_j - i_{l}|})^{\frac{p_j}{m(m-1)}} \lesssim |\mathbf{n}|$ since 
\begin{eqnarray*}
  \min(1, 2^{|\mathbf{n}| - |i_j - i_{l}|}) = \left\{
    \begin{array}{lll}
      2^{|\mathbf{n}| + i_j - i_l}, &i_l>i_j+|\mathbf{n}|;\\
      1, &i_j-|\mathbf{n}|\leq i_l \leq i_j+|\mathbf{n}|;\\
      2^{|\mathbf{n}| - i_j + i_l}, &i_l<i_j - |\mathbf{n}|.
    \end{array}
  \right.
\end{eqnarray*}
Therefore $\| \mathcal{M}_\mathbf{n}(b_1, \dots, b_m) \|_{L^p(\R^d\setminus\mathcal{E})}$ is bounded by a constant multiple of
\begin{align}\label{ineq-mn bad esti}
  |\mathbf{n}|^{\frac{m}{p}} \prod_{j=1}^m \Bigl( \sum_{i_j\in\Z}\|b_j^{i_j}\|_{p_j}^{p_j} \Bigr)^{\frac{1}{p_j}}
  =|\mathbf{n}|^{\frac{m}{p}} \prod_{j=1}^m \|b_j\|_{p_j}.
\end{align}
With help of \eqref{ineq-mn bad esti}, we finally estimate the levelset of $\mathcal{M}_\mathbf{n}(b_1, \dots, b_m)$
\begin{align*}
  \meas\Bigl(\{x : \mathcal{M}_{\mathbf{n}}(b_1, \dots, b_m)(x) > 2^{-m}\lambda\}\Bigr)
  &\lesssim \lambda^{-p} \| \mathcal{M}_{\mathbf{n}}(b_1, \dots, b_m)\|_{L^p(\R^d\setminus\mathcal{E})}^p\\
  &\leq \lambda^{-p} |\mathbf{n}|^{m} \prod_{j=1}^m \|b_j\|_{p_j}^p.
\end{align*}
Since $\|b_j\|_{p_j} \lesssim 1$, we obtain \eqref{ineq-levelset esti} for $b_1, \dots, b_m$.

\subsubsection{Estimates for other terms}

The cases $\mathcal{M}_\mathbf{n}(g_1, \dots, g_m)$, $\mathcal{M}_\mathbf{n}(g_1, \dots,g_{m-1}, b_m )$, $\dots$, $\mathcal{M}_\mathbf{n}(b_1, \dots, b_{m-1}, g_m)$ follow from simplified arguments given in \ref{sssec-bad}. 
We first consider the cases except for $\mathcal{M}_{\mathbf{n}}(g_1, \dots, g_m)$.
Thus we define for $\alpha + \beta = m$ and $1\leq \alpha, \beta \leq m-1$
\begin{align}
  \mathcal{M}_{\mathbf{n}}(\mathbf{g}^\alpha, \mathbf{b}^{\beta})
  =\mathcal{M}_{\mathbf{n}}(g_1, \dots, g_\alpha, b_{\alpha+1}, \dots, b_m).
\end{align}
Note that one can modify the proofs of Lemmas~\ref{lem-decay for b}, \ref{lem-sum} to obtain $\mathbf{b}^\beta$-analogue. 
Let  $(0, \dots, 0, \frac{1}{r_{\alpha+1}}, \dots, \frac{1}{r_m}) \in \text{conv}(\mathcal{V}_{\kappa})$ with $\frac{1}{p} = \sum_{\nu=\alpha+1}^m \frac{1}{r_\nu}$.
Then the proofs of Lemma~\ref{lem-decay for b} yields that
\begin{equation}\label{ineq-b beta 1}
\begin{aligned}
  &\| \mathcal{A}_\ell^\mathbf{n}(\mathbf{g}^\alpha, \mathbf{b}^\beta) \|_{L^p(\R^d\setminus\mathcal{E})}\\
  &\lesssim \prod_{\mu=1}^\alpha \|g_\mu\|_{L^\infty} \min_{\nu=\alpha+1, \dots, m} \min(1, 2^{(n_\nu+\ell - i_\nu)(1+\frac{d}{r_\nu'})}, 2^{i_\nu - \ell}) \prod_{\nu=\alpha+1}^m \|b_\nu^{i_\nu}\|_{r_\nu}.
\end{aligned}
\end{equation}
We have 
\begin{equation}\label{ineq-b beta 2}
\begin{aligned}
    &\sum_{\ell\in\Z} \min_{i_{\alpha+1}, \dots, i_m} \min(1, 2^{(n_\nu+\ell - i_\nu)(1+\frac{d}{r_\nu'})}, 2^{i_\nu - \ell}) \\
    &\lesssim |\mathbf{n}| \prod_{\alpha+1 \leq \nu, \nu'\leq m, j\not=j'} \min(1, 2^{|\mathbf{n}| - |i_\nu - i_{\nu'}|})^{\frac{1}{\beta(\beta-1)}}.
\end{aligned}
\end{equation}
With help of \eqref{ineq-b beta 1} and \eqref{ineq-b beta 2}, we estimate $\|\mathcal{M}_{\mathbf{n}}(\mathbf{g}^\alpha, \mathbf{b}^{\beta})\|_{L^p(\R^d \setminus \mathcal{E})}$ as following:
\begin{align}\label{ineq-Mn beta}
  \|\mathcal{M}_{\mathbf{n}}(\mathbf{g}^\alpha, \mathbf{b}^{\beta})\|_{L^p(\R^d \setminus \mathcal{E})}
  \lesssim |\mathbf{n}|^{\frac{\beta}{p}} \prod_{\mu=1}^\alpha \|g_\mu\|_{L^\infty} \times \prod_{\nu=\alpha+1}^m \|b_\nu\|_{L^{r_\nu}}.
\end{align}
Since $\supp(b_\nu) \subset \cup_{\gamma}Q_{\nu, \gamma}$ and $\sum_{\gamma} \meas(Q_{\nu, \gamma}) \lesssim \lambda^{-p}$ due to \eqref{CZ-bad}, \eqref{CZ-bad esti}, 
the right-hand side of \eqref{ineq-Mn beta} is bounded by a constant multiple of 
\begin{align}\label{ineq-230220 1436}
  |\mathbf{n}|^\frac{\beta}{p} \lambda^{\sum_{\mu=1}^\alpha \frac{p}{p_j}} \times \lambda^{-p(\sum_{\nu=\alpha+1}^m \frac{1}{r_\nu} - \frac{1}{p_\nu})}
  = |\mathbf{n}|^\frac{\beta}{p}.
\end{align}
Here the left-hand side of \eqref{ineq-230220 1436} is a consequence of H\"older's inequality on $\|b_\nu\|_{L^{r_\nu}}$.
Finally, by making use of \eqref{ineq-Mn beta} and \eqref{ineq-230220 1436}, we have
\begin{align*}
  \meas\Bigl(\{x : \mathcal{M}_{\mathbf{n}}(\mathbf{g}^\alpha, \mathbf{b}^\beta)(x) > 2^{-m}\lambda\}\Bigr)
  &\lesssim \lambda^{-p} \| \mathcal{M}_{\mathbf{n}}(\mathbf{g}^\alpha, \mathbf{b}^\beta)\|_{L^p(\R^d\setminus\mathcal{E})}^p\\
  &\leq \lambda^{-p} |\mathbf{n}|^{\beta}.
\end{align*}

For $\mathcal{M}_n(g_1, \dots, g_m)$, we simply choose $C < 2^{-1}$ in \eqref{CZ-height} so that
$$
  |\mathcal{M}_n(g_1, \dots, g_m)| \leq \prod_{j=1}^m \|g_j\|_{L^\infty} \leq C^m \lambda < 2^{-m}\lambda.
$$
Thus we have 
$$
\meas\Bigl(\{x : \mathcal{M}_{\mathbf{n}}(g_1, \dots, g_m)(x) > 2^{-m}\lambda\}\Bigr) = 0.
$$

\subsubsection{Proof of Lemma~\ref{lem-decay for b}}

  For simplicity, let $n_j+\ell = \tau$ and $b_j^{i_j} =b = \sum_{Q} b_Q$ where $Q = Q_{j, \gamma}$ whose sidelength is $2^{-i}$.
  Then thanks to Proposition~\ref{prop-qbanach esti}, it suffices for the first and second term in the minimum to show that
  \begin{align}
    \| P_\tau b_Q \|_{p} \lesssim \min(1, (2^\tau s(Q))^{1+\frac{d}{p'}}) \| b_Q\|_p.
  \end{align}
  The first term, $1$, is directly given by the fact that $\|\psi_\tau\|_{1}=1$ and Young's inequality.

  For the second term, we make use of the vanishing property of $b_Q$.
  Let $c_Q$ be the center of $Q$.
  \begin{align*}
    P_\tau b_Q(x) 
    &= \int_{\R^d} (\psi_\tau(x-y) - \psi_\tau(x-c_Q)) b_Q(y)~~\mathrm{d}y\\
    &= \int_{\R^d} \int_0^1 \langle \nabla_y(\psi_\tau)(x-c_Q - t(y-c_Q)), y-c_Q \rangle~~\mathrm{d}t~ b_Q(y)~~\mathrm{d}y.
  \end{align*}
  Since $\psi$ is of Schwartz class, it follows that $\frac{1}{|y-c_Q|}\int_0^1 \langle \nabla_y(\psi_\tau)(x-c_Q - t(y-c_Q)), y-c_Q \rangle~~\mathrm{d}t$ is bounded by a constant multiple of
  $$
    \frac{2^{\tau(d+1)}}{(1+2^\tau|x-c_Q - t(y-c_Q)|)^N},
  $$
  for any $ N>0$.
  Thus we apply Minkowski's integral inequality  to obtain
  \begin{align*}
    \| P_\tau b_Q\|_{L^p(\R^d)}
    &\lesssim \Big(\int_{\R^d} \frac{2^{\tau(d+1)p}}{(1+2^\tau|x|)^{pN}}~~\mathrm{d}x \Big)^{\frac{1}{p}}\times \int_{Q} |y-c_Q||b_Q(y)|~~\mathrm{d}y \\
    &\leq 2^{\tau(d+1)}2^{-\tau \frac{d}{p}} s(Q) s(Q)^{d-\frac{d}{p}} \|b_Q\|_p.
  \end{align*}
  This establishes
  $$
  \| P_\tau b_Q\|_{L^p(\R^d)} \lesssim 2^{\tau(1+\frac{d}{p'})} s(Q)^{1+\frac{d}{p'}} \|b_Q\|_p.
  $$
  Therefore we have
  \begin{equation}\label{ineq-orthog}
  \begin{aligned}
    \| P_\tau b\|_{L^p(\R^d)} 
    &= \Bigl(\sum_Q \|P_\tau b_Q\|_p^p\Bigr)^{\frac{1}{p}}\\
    &\lesssim (2^{\tau}s(Q))^{1+\frac{d}{p'}} \Bigl(\sum_Q \|b_Q\|_p^p\Bigr)^{\frac{1}{p}} = (2^{\tau}s(Q))^{1+\frac{d}{p'}} \|b\|_p.
  \end{aligned}
  \end{equation}
  The first and the last equalities follow from the disjointness of $Q$'s. This gives a decay estimate when $n_j+\ell<i_j$.

  Lastly, we assume that $\ell>i_j$ so that for $x\in (C_\mathcal{S}Q)^\complement$ and $z\in Q$, we have $\dist(x-2^{-\ell}y, z) \geq s(Q)$ uniformly in $y\in\mathcal{S}$, because we choose $C_\cS = 5 \max(1, \diam(\cS))$.
  Thus it follows that
  \begin{align*}
    P_{n_j+\ell}b_{j,\gamma}^{i_j}(x-2^{-\ell}\Theta_j y) = \widetilde{P}_{n_j+\ell}b_{j,\gamma}^{i_j}(x-2^{-\ell}\Theta_j y),
  \end{align*}
  where the kernel of $\widetilde{P}_\tau$ is given by 
  $$
    \psi_\tau(y) \mathbbm{1}_{|y|\geq s(Q)}(y).
  $$
  Therefore, we obtain with help of \eqref{ineq-orthog} that 
  \begin{align*}
    \| \mathcal{A}_\ell^\mathbf{n}(b_{1}^{i_1}, \dots, b_{m}^{i_m}) \|_{L^p(\R^d\setminus\mathcal{E})}
    \lesssim \min_{j=1, \dots, m} \|\widetilde{P}_{n_j +\ell} \|_{p\to p} \prod_{j=1}^m \|b_{j}^{i_j}\|_{p_j}.
  \end{align*}
  Observe that for any $N>0$
  \begin{align*}
    \|\widetilde{P}_{n_j +\ell} \|_{p\to p} \leq  \int_{\R^d} |\psi_{n_j+\ell - i_j}(x)| \mathbbm{1}_{|x|\geq 1}(x)~~\mathrm{d}x \lesssim 2^{-N(n_j+\ell -i_j)}.
  \end{align*}
  Thus we have $\|\widetilde{P}_{n_j +\ell} \|_{p\to p} \leq 2^{-\ell +i_j}$ regardless of $n_j\geq0$.
  This proves the lemma.

\subsubsection{Proof of Lemma~\ref{lem-sum}}

It suffices to show that for any $i_1$, $i_2$ 
\begin{equation}\label{ineq-min sum}
\begin{aligned}
    \sum_{\ell\in\Z} \min_{i_1, i_2} \min(1, 2^{(n_j+\ell - i_j)(1+\frac{d}{p_j'})}, 2^{i_j - \ell}) 
    \lesssim |\mathbf{n}| \min(1, 2^{|\mathbf{n}| - |i_1 - i_2|}).
\end{aligned}
\end{equation}
Note that 
\begin{eqnarray*}
  \min(1, 2^{(n_j+\ell - i_j)(1+\frac{d}{p_j'})}, 2^{i_j - \ell})  \leq \left\{
    \begin{array}{lll}
      2^{i_j -\ell}, &i_j< \ell;\\
      1, &i_j-|\mathbf{n}|\leq \ell \leq i_j;\\
      2^{(|\mathbf{n}| -i_j + \ell)(1+\frac{d}{p_j'})}, &\ell<i_j - |\mathbf{n}|.
    \end{array}
  \right.
\end{eqnarray*}
When $i_1 \sim i_2$, then the left side of \eqref{ineq-min sum} is bounded by a costant multiple of $|\mathbf{n}|$.
Thus we consider the case of $i_2$ is greater than $i_1+|\mathbf{n}|$.
Since $i_2 >i_1+|\mathbf{n}|$, it follows that for $i_1-|\mathbf{n}| \leq \ell \leq i_1$
\begin{align}\label{i2 i1 1}
  \min_{i_1, i_2} \min(1, 2^{(n_j+\ell - i_j)(1+\frac{d}{p_j'})}, 2^{i_j - \ell}) \leq 2^{(|\mathbf{n}| -i_2 +\ell)(1+\frac{d}{p_j'})}\leq 2^{(|\mathbf{n}| -|i_1 - i_2|)(1+\frac{d}{p_j'})}.
\end{align}
Similarly, we have for $i_2-|\mathbf{n}| \leq \ell \leq i_2$
\begin{align}\label{i2 i1 2}
  \min_{i_1, i_2} \min(1, 2^{(n_j+\ell - i_j)(1+\frac{d}{p_j'})}, 2^{i_j - \ell}) \leq 2^{i_1 -\ell}\leq 2^{|\mathbf{n}| -|i_1 - i_2|}.
\end{align}
One can obtain the same bounds when $i_1>i_2 +|\mathbf{n}|$ by changing roles of $i_1, i_2$ in \eqref{i2 i1 1}, \eqref{i2 i1 2}.
Therefore we conclude that
$$
  \sum_{\ell\in\Z} \min_{i_1, i_2} \min(1, 2^{(n_j+\ell - i_j)(1+\frac{d}{p_j'})}, 2^{i_j - \ell}) 
  \lesssim |\mathbf{n}| \min(1, 2^{|\mathbf{n}| - |i_1 - i_2|}).
$$

\subsection{Proof of Lemma~\ref{lem-smoothing}}
  By the assumption \eqref{ineq_reg_cM'}, for $1\leq r_1, \dots, r_m<\infty$ with $1=\sum_{j=1}^m \frac{1}{r_j}$, we have
  \begin{align}
    \| \mathcal{A}_\mathcal{S}^\Theta(P_{n_1}f_1, \dots, P_{n_m}f_m) \|_{L^1}
    \lesssim 2^{-\delta|\mathbf{n}|} \prod_{j=1}^m \|P_{n_j} f_j \|_{L^{r_j}},\quad \delta>0.
  \end{align}
  We make use of the following scaling invariance of $\mathcal{A}_\mathcal{S}^\Theta$:
  \begin{align*}
    \Big\| \int_{\mathcal{S}} \prod_{j=1}^m P_{\ell + n_j}f_j(x - 2^{-\ell}\Theta_j y)~~\mathrm{d}\sigma(y) \Big\|_{L^1(\mathrm{d}x)}
    =2^{-\ell d} \| \mathcal{A}_\mathcal{S}^\Theta(P_{n_1}f_{1,-\ell}, \dots, P_{n_m}f_{m,-\ell}) \|_{L^1},
  \end{align*}
  where $f_{j,-\ell}(x) = f_j(2^{-\ell}x)$.
  Then it follows that 
  \begin{align*}
    \| \mathcal{S}_{\mathbf{n}}(\rF)\|_{L^1} 
    &\leq \sum_{\ell\in\Z} 2^{-\ell d} \| \mathcal{A}_\mathcal{S}^\Theta(P_{n_1}f_{1,-\ell}, \dots, P_{n_m}f_{m,-\ell}) \|_{L^1}\\
    &\lesssim \sum_{\ell\in\Z} 2^{-\ell d} 2^{-\delta|\mathbf{n}|} \prod_{j=1}^m \| P_{n_j}f_{j, -\ell} \|_{L^{r_j}}\\
    &= \sum_{\ell\in\Z} 2^{-\ell d} 2^{-\delta|\mathbf{n}|} \prod_{j=1}^m 2^{\frac{\ell d}{r_j}}\| P_{n_j+\ell}f_{j} \|_{L^{r_j}}\\
    &= \sum_{\ell\in\Z}  2^{-\delta|\mathbf{n}|} \prod_{j=1}^m \| P_{n_j+\ell}f_{j} \|_{L^{r_j}}.
  \end{align*}
  We apply H\"older's inequality to the last line and obtain
  \begin{align*}
    \| \mathcal{S}_{\mathbf{n}}(\rF)\|_{L^1} 
    \lesssim 2^{-\delta|\mathbf{n}|} \prod_{j=1}^m \Bigl(\sum_{\ell\in\Z}\| P_{n_j+\ell}f_{j} \|_{L^{r_j}}^{r_j}\Bigr)^{\frac{1}{r_j}}.
  \end{align*}
  Note that $(\sum_j \|P_jf\|_p^p)^{1/p} \lesssim \|f\|_p$ for $p\geq2$, which gives
  $$
    \| \mathcal{S}_{\mathbf{n}}(\rF)\|_{L^1} \lesssim 2^{-\delta|\mathbf{n}|} \prod_{j=1}^m \|f_j\|_{L^{r_j}}.
  $$
  This proves the lemma.

  \section{Proof of Theorem~\ref{thm-lac-cM}}

Recall that for an $(md-1)$-dimensional hypersurface $\Sigma$ in $\R^{md}$ with $\kappa$ non-vanishing principal curvatures and $\kappa>(m-1)d$, we define $\textsl{A}_{\Sigma}(\rF)(x)$ as following:
\begin{align}
  \int_{\Sigma}\prod_{j=1}^m f_j(x - y_j)~~\mathrm{d}\sigma_\Sigma(\textsl{y}),\quad \textsl{y} = (y_1, \dots, y_m) \in \R^{md}.
\end{align}

By making use of the dyadic decomposition of Section~\ref{sec_thm_lac} satisfying \eqref{proj<ell}, \eqref{proj_ell}, \eqref{proj_identity}, and \eqref{proj_mlinear_id},
we define the following quantities similar to \eqref{A_ell_F}, \eqref{M_n_F}, \eqref{S_n_F}:
\begin{align}
	 \textsl{A}_{\ell}^{\alpha,\tau}(\rF)(x)
    &:=\int_{\Sigma} \Bigl(\prod_{i=1}^\alpha P_{<\ell}f_{\tau(i)}(x-2^{-\ell}y_{\tau(i)}) \Bigr) \Bigl( \prod_{i=\alpha+1}^m f_{\tau(i)} (x - 2^{-\ell}y_{\tau(i)})\Bigr)~~\mathrm{d}\sigma(\textsl{y}),\label{scrA_ell_F}\\
    \widetilde{\textsl{A}}_{\ell}^{\alpha,\tau}(\rF)(x)
    &:=\int_{\Sigma} \Bigl(\prod_{i=1}^\alpha f_{\tau(i)}(x-2^{-\ell}y_{\tau(i)}) \Bigr) \Bigl( \prod_{i=\alpha+1}^m  P_{<\ell}f_{\tau(i)} (x - 2^{-\ell}y_{\tau(i)})\Bigr)~~\mathrm{d}\sigma(\textsl{y}),\\
	 \textsl{M}_{\bf{n}}(\rF)
    &:=\sup_{\ell\in\mathbb{Z}} \Big|\int_{\Sigma} \prod_{j=1}^m P_{\ell+n_j} f_j(x - 2^{-\ell}y_j)~\mathrm{d}
    \sigma(\textsl{y}) \Big|,\label{scrM_n_F}\\
	 \textsl{S}_{\bf{n}}(\rF)
    &:=\sum_{\ell\in\mathbb{Z}}\Big|\int_{\Sigma} \prod_{j=1}^m P_{\ell+n_j} f_j(x - 2^{-\ell} y_j)~\mathrm{d}
    \sigma(\textsl{y}) \Big|.\label{scrS_n_F}
\end{align}
Therefore, the lacunary maximal operator $\mathfrak{M}_{\Sigma}$ is bounded by a constant mutiple of 
\begin{align}\label{ineq_230303_1415}
  \sum_{\alpha=1}^m \sum_{\tau\in S_m} \sup_{\ell\in\Z}( |\textsl{A}_\ell^{\alpha,\tau}(\rF)| + | \widetilde{\textsl{A}}_{\ell}^{\alpha,\tau}(\rF)|) + \sum_{\mathbf{n}\in \mathbb{N}_0^m}  \textsl{M}_{\mathbf{n}}(\rF),\quad \mathbb{N}_0 = \mathbb{N} \cup \{0\}.
\end{align}
As in the previous section, instead of the first summation in \eqref{ineq_230303_1415} it suffices to consider estimates for $\textsl{A}_\ell^{\alpha}(\rF)$, which is given by
$$
	\textsl{A}_{\ell}^{\alpha}(\rF)(x)
    :=\int_{\Sigma} \Bigl(\prod_{j=1}^\alpha P_{<\ell}f_{j}(x-2^{-\ell}y_{j}) \Bigr) \Bigl( \prod_{j=\alpha+1}^m f_{j} (x - 2^{-\ell}y_{j})\Bigr)~~\mathrm{d}\sigma(\textsl{y}).
$$
Then the proof will be completed by combination of the following lemmas and an induction argument which is slightly different from the argument in Section~\ref{sec_thm_lac}:
\begin{lemma}\label{lem_m2_frakM}
  Let $F = (f_1, f_2, 1, \dots, 1)$ and $\alpha=1$. 
  Then we have
  $$
     \textsl{A}_\ell^\alpha(\rF)(x) \leq M_{HL}(f_1)(x) \times {M}_{\Sigma}(f_2)(x),
  $$
  where $ {M}_{\Sigma}(f)(x) = \sup_{\ell\in\Z}|\int_{\Sigma} f(x-2^{-\ell}y_2)~\mathrm{d}\sigma(\textsl{y})|$.
\end{lemma}
The proof of Lemma~\ref{lem_m2_frakM} is same with that of Lemma~\ref{lem-m2}, so we omit it.
Note that $M_{HL}, {M}_{\Sigma}$ are bounded on $L^p$ for $p\in(1, \infty]$, 
hence we need the boundedness of the second term in \eqref{ineq_230303_1415}.

\begin{lemma}\label{lem-qba esti_scrM}
  Let $\mathbf{n} \in \mathbb{N}^m$ and $\frac{m+1}{2} \leq \frac{1}{p} < \frac{2d+\kappa}{2d}$.
  For $p_j \in [1,2], j=1\cdots,m$ with $\sum_{j=1}^m \frac{1}{p_j} = \frac{1}{p}$, we have
  $$
    \|  \textsl{M}_{\mathbf{n}}(\rF) \|_{L^{p, \infty}} \leq C(1+|\mathbf{n}|^{m}) \prod_{j=1}^m \|f_j\|_{L^{p_j}}.
  $$
\end{lemma}
The proof of Lemma~\ref{lem-qba esti_scrM} is to repeat the proof of Lemma~\ref{lem-qba esti}. 
The only difference occurs in showing Lemma~\ref{lem-decay for b} in terms of $ \textsl{A}_\ell^\mathbf{n}$ which corresponds to $\mathcal{A}_\ell^\mathbf{n}$, 
since $ \textsl{A}_\ell^\mathbf{n}$ is an average over $\Sigma$ which is $(md-1)$-dimensional and each $f_j$ depends on $x-y_j$ not $x-\Theta_j y$.
This difference is harmless, however, because only the compactness of $\cS$ does matter in the proof of Lemma~\ref{lem-decay for b} and $\Sigma$ is a compact hypersurface.
On the other hand, the range $\frac{m+1}{2} \leq \frac{1}{p} < \frac{2d+\kappa}{2d}$ follows from Proposition~\ref{prop_qbesti_cM}.

\begin{lemma}\label{lem-smoothing_scrS}
  Let $\mathbf{n} \in \mathbb{N}^m$ and $1= \sum_{j=1}^m \frac{1}{r_j}$. Then it holds that
  $$
    \| \textsl{S}_{\mathbf{n}}(\rF)\|_{L^1} \lesssim 2^{-\delta|\mathbf{n}|} \prod_{j=1}^m \|f_j\|_{L^{r_j}}.
  $$
\end{lemma}
The proof of Lemma~\ref{lem-smoothing_scrS} is the same with that of Lemma~\ref{lem-smoothing} together with \eqref{ineq_reg_cM}, so we omit it.

Due to $ \textsl{M}_\mathbf{n} \leq  \textsl{S}_\mathbf{n}$ by definitions \eqref{scrM_n_F}, \eqref{scrS_n_F}, it follows from interpolation between Lemmas~\ref{lem-qba esti_scrM} and \ref{lem-smoothing_scrS} that
\begin{align}\label{ineq_230406_1923}
  \|  \textsl{M}_\mathbf{n}(\rF) \|_{L^p} \lesssim 2^{-\delta'|\mathbf{n}|} \prod_{j=1}\|f_j\|_{L^{p_j}},\quad \delta'>0,
\end{align}
whenever $\frac{1}{p_1}+\cdots+\frac{1}{p_m}=\frac{1}{p} <\frac{2d+\kappa}{2d}$.
It should be noted that Lemmas~\ref{lem-qba esti_scrM} and \ref{lem-smoothing_scrS} are still valid for $F = (f_1, \dots, f_N, 1, \dots, 1)$ with $m$ replaced by $N$ and taking $L^\infty$ norms for $1$'s.
That is, for $\rF=(f_1, \dots, f_N, 1,\dots, 1)$ we have
\begin{align}
	 \|  \textsl{M}_{\mathbf{n}}(\rF) \|_{L^{p, \infty}} &\leq C(1+|\mathbf{n}|^{N}) \prod_{j=1}^N \|f_j\|_{L^{p_j}},\quad \frac{1}{p} = \frac{1}{p_1} + \cdots \frac{1}{p_N},\\
	  \| \textsl{S}_{\mathbf{n}}(\rF)\|_{L^1} &\lesssim 2^{-\tilde{\delta}|\mathbf{n}|} \prod_{j=1}^N \|f_j\|_{L^{r_j}},\quad 1 = \frac{1}{r_1} + \cdots \frac{1}{r_N},
\end{align}
for some $\tilde{\delta}>0$.
Thus instead of \eqref{ineq_230406_1923}, we have for $F= (f_1, \dots, f_N, 1, \dots, 1)$
$$
	\|  \textsl{M}_\mathbf{n}(\rF) \|_{L^p} \lesssim 2^{-\tilde{\delta}'|\mathbf{n}|} \prod_{j=1}^N \|f_j\|_{L^{p_j}},\quad \text{for some $\tilde{\delta}'>0$}.
$$
Since $2^{-\delta'|\mathbf{n}|}$ is summable over $\mathbf{n}\in\mathbb{N}_0^{m}$, this together with Lemma~\ref{lem_m2_frakM} proves the theorem for $F= (f_1, f_2, 1, \dots, 1)$.

For the induction, we assume that Theorem~\ref{thm-lac-cM} holds for $F= (f_1, \dots, f_N, 1, \dots, 1)$ for $N=2, \dots, m-1$ with $1/p=1/p_1 +\cdots+1/p_N$.
Note that we have shown $N=2$-case.
By the assumption, we have the following lemma:
\begin{lemma}\label{lem-induction_frakM}
  For $\alpha=1, \dots, m$, we have
  $$
     \textsl{A}_\ell^\alpha(\rF)(x) 
    \lesssim 
    \prod_{\mu=1}^\alpha M_{HL}(f_\mu)(x) \times 
    \sup_{\ell\in\Z}\Big|\int_{\Sigma} \prod_{\nu=\alpha+1}^m f_\nu(x - 2^{-\ell} y_\nu)~\mathrm{d}\sigma_\Sigma(\textsl{y})\Big|.
  $$
  Moreover, if we assume that Theorem~\ref{thm-lac-cM} is true for $F= (f_1, \dots, f_N, 1, \dots, 1)$ with $N=2, \dots, m-1$ and $1/p=1/p_1 +\cdots+1/p_N$, then it follows that
  $$
  \sup_{\ell\in\Z}\Big|\int_{\Sigma} \prod_{\nu=\alpha+1}^m f_\nu(x - 2^{-\ell} y_\nu)~~\mathrm{d}\sigma_\Sigma(\textsl{y})\Big|
  $$
  satisfies multilinear estimates of Theorem~\ref{thm-lac-cM} for $(m-\alpha)$-linear operators.
\end{lemma}
\begin{proof}
The first assertion of the lemma follows directly by the proof of Lemma~\ref{lem-m2}.
For the second assertion, recall that $|\widehat{\mathrm{d}\sigma_\Sigma}(\xi)|\lesssim (1+|\xi|)^{-\kappa/2}$ for $(m-1)d<\kappa\leq md-1$.
Theorem~\ref{thm-lac-cM} holds for $(m-\alpha)$-linear maximal averages when $\kappa>(m-\alpha-1)d$, which is already affirmative.
Thus the assertion holds from the assumption that Theorem~\ref{thm-lac-cM} is true for $F= (f_1, \dots, f_N, 1, \dots, 1)$ with $N=2, \dots, m-1$ and $1/p=1/p_1 +\cdots+1/p_N$.
\end{proof}
Since we already proved Lemmas~\ref{lem-qba esti} and \ref{lem-smoothing} for general $m$, Theorem~\ref{thm-lac-cM} for $m$-linear operators holds under the assumption that $N=2, \dots, m-1$ cases hold.
This closes the induction hence proves the theorem.

We end this section by suggesting the proof of Remark~\ref{bilinearspherical}.
\begin{proof}[Proof of Remark \ref{bilinearspherical}]
	Note that, for dimension $d=1$, proof of this remark is already given in \cite{Christ_Zhou2022}.
	Although the proof for $d\geq 2$  case is given in \cite{Bo_Fo2023}, we present a different proof  by exploiting ideas of \cite{Christ_Zhou2022} to prove Remark \ref{bilinearspherical} in higher dimensions $d\geq2$. 
	In fact, the proof follows from Theorem~\ref{thm-lac-cM} with minor modifications in the Lemma \ref{lem-qba esti_scrM} and \ref{lem-smoothing_scrS}. Indeed, note that in \cite{Iosevich} the authors proved $L^{1}\times L^{1}\rightarrow L^{1/2}$ estimate of the bilinear spherical average $\textsl{A}^{1}_{\mathbb{S}^{2d-1}}$. Using this estimate in Lemma \ref{lem-qba esti_scrM} we get 
	 $$
	\|  \textsl{M}_{\mathbf{n}}(\rF) \|_{L^{\frac{1}{2}, \infty}} \leq C(1+|\mathbf{n}|^{2}) \prod_{j=1}^2 \|f_j\|_{L^{1}}.
	$$
	Further, using the estimate in Lemma \ref{lem-smoothing_scrS} with $\Sigma=\mathbb{S}^{2d-1}$, we get 
	 $$
	\| \textsl{S}_{\mathbf{n}}(\rF)\|_{L^1} \lesssim 2^{-\delta|\mathbf{n}|} \prod_{j=1}^m \|f_j\|_{L^{2}}
	$$ for some $\delta>0$. The rest of the proof follows by imitating the machinery of Theorem \ref{thm-lac-cM}.
\end{proof}

\section*{Acknowledgement}

All three authors have been partially supported by NRF grant no. 2022R1A4A1018904 funded by the Korea government(MSIT) .
They are supported individually by NRF no. RS-2023-00239774(C. Cho), no. 2021R1C1C2008252(J. B. Lee), and BK21 Postdoctoral fellowship of Seoul National University(K. Shuin).
The authors are sincerely grateful to referees for their comments which make this article more readable.

\end{document}